\documentclass[11pt,letterpaper]{article}
\usepackage[utf8]{inputenc}
\usepackage{amsmath,wasysym,colortbl,amsfonts,amssymb}
\usepackage{amsthm,graphicx,mathtools,array,multirow,longtable}
\usepackage{thm-restate}
\usepackage[font=small,labelfont=bf]{caption}

\DeclarePairedDelimiter{\ceil}{\lceil}{\rceil}
\DeclarePairedDelimiter\floor{\lfloor}{\rfloor}
\DeclareMathOperator{\crg}{cr}

\DeclareMathOperator{\lc}{lcr}

\DeclareMathOperator{\conv}{conv}
\DeclareMathOperator{\gen}{gen}
\DeclareMathOperator{\nonedge}{\partial _{non-edge}}
\DeclareMathOperator{\bd}{\partial\conv}

\newcolumntype{C}{ >{\centering\arraybackslash} m{4cm} }
\newcolumntype{D}{ >{\centering\arraybackslash} m{1cm} }
\newcolumntype{E}{ >{\centering\arraybackslash} m{2cm} }
\usepackage{float}
\usepackage[left=2.5cm,right=2.5cm,top=2.5cm,bottom=2.5cm]{geometry}
\usepackage{amstext}
\usepackage{array}
\usepackage{authblk}

\newtheorem{thm}{Theorem}
\newtheorem{prop}[thm]{Proposition}
\newtheorem{lem}[thm]{Lemma}
\newtheorem{cor}[thm]{Corollary}
\newtheorem{conj}[thm]{Conjecture}

\def\crg{\mathop{\rm cr}}
\def\PRRany{\rm 1}
\def\PNNd{\rm 2}
\def\PNdN{\rm 3}
\def\PNa{\rm 4}
\def\PanyLa{\rm 5}
\def\PaRRa{\rm 6}
\def\PN{\rm 7}
\def\PaLL{\rm 8}
\def\PLLa{\rm 9}
\def\PanyL{\rm 10}
\def\PV{\rm 11}

\title{Book crossing numbers of the complete graph\\ and small local convex crossing numbers}
\author
{Bernardo M. \'Abrego\footnote{California State University, Northridge, [bernardo.abrego, silvia.fernandez,julia.kinzel]@csun.edu.}, 
Julia Dandurand$^{*\dag}$, 
Silvia Fern\'andez-Merchant$^*$\thanks{Supported by the NSF grant DMS-1400653.}, \\
Evgeniya Lagoda\footnote{Berlin Mathematical School, lagoda@math.tu-berlin.de}, 
Yakov Sapozhnikov\footnote{University of Idaho, yakovs@uidaho.edu.} $^\dag$}

\begin{document}
\maketitle
\begin{abstract}
A \emph{$ k $-page book drawing} of a graph $ G $ is a drawing of $ G $ on $ k $ halfplanes with a line $ l $ as a common boundary such that the vertices are located on $ l $ and the edges cannot cross $ l $. The \emph{$ k $-page book crossing number} of the graph $ G $, denoted by $ \nu_k(G) $, is the minimum number of edge-crossings over all $ k $-page book drawings of $ G $. This paper improves previous results on $ k $-page book crossing numbers of the complete graph $ K_n $. We determine $ \nu_k(K_n) $ whenever $ 2 < n/k \leq 3 $ and improve the lower bounds on $ \nu_k(K_n) $ for all $ k\geq 14 $. Our proofs rely on bounding the number of edges in convex graphs with small local crossing numbers. In particular, we determine the maximum number of edges that a convex graph with at most $ \ell $ crossings per edge can have for any $ \ell\leq 4 $.
\end{abstract}
\section{Introduction}
A \emph{$ k $-page book} is an object in the space formed by $ k $ halfplanes with common boundary. The common boundary is a line called the \emph{spine} of the book and the halfplanes are called the \emph{pages} of the book. A \emph{$ k $-page book drawing} of a graph $ G $ is a drawing of $ G $ in a $ k $-page book so that the vertices lay on the spine of the book and the edges do not cross the spine. The \emph{$ k $-page book crossing number} of the graph $ G $, denoted by $ \nu_k(G) $, is the minimum number of edge-crossings over all $ k $-page book drawings of $ G $. Book crossing numbers have been studied in relation to their applications in VLSI designs \cite{CLR87, L83}. We are concerned with the $ k $-page book crossing number of the complete graph $ K_n $. 

In 1964, Bla\v{z}ek and Koman \cite{BC64} described $ k $-page book drawings of $ K_n $ with few crossings and proposed the  problem of determining $ \nu_k(K_n) $. They only described their construction in detail for $ k=2 $, explicitly pointing out the exact number of crossings in their construction for $ k=2 $ and $ 3 $, and indicated that their construction could be generalized to larger values of $ k $, implicitly conjecturing that their construction achieves $ \nu_k(K_n) $. In 1994, Damiani, D'Antona, and Salemi \cite{DaDASa} described constructions in detail using adjacency matrices, but did not explicitly compute their exact crossing numbers. Two years later, Shahrokhi, S\'{y}kora, Sz\'{e}kely, and Vrt'o \cite{SSSV96} provided a geometric description of $ k $-page book drawings of $ K_n $ and bounded their number of crossings, showing that
\[
\nu_k(K_n)\leq \frac{2}{k^2}\left(1-\frac{1}{2k}\right)\binom{n}{4}+\frac{n^3}{2k}.
\]
In 2013, De Klerk, Pasechnik, and Salazar  (see Proposition 5.1 in \cite{DPS13}) gave another construction and computed its exact number of crossings using the geometric approach in \cite{SSSV96}. It can be expressed as 
\begin{equation}\label{eq:zk1}
Z_k(n):=(n \bmod k)\cdot F\left( \left\lfloor {\dfrac{n}{k}}\right\rfloor+1,n  \right)+(k-(n \bmod k))\cdot F\left( \left\lfloor {\dfrac{n}{k}}\right\rfloor,n  \right),
\end{equation}
where
\begin{equation}\label{eq:zk2}
F(r,n):=\frac{r}{24}(r^2-3r+2)(2n-3-r).
\end{equation}
Then $ \nu_k (K_n)\leq Z_k(n). $
All the constructions in \cite{DaDASa}, \cite{SSSV96}, and \cite{DPS13} generalize the original Bla\v{z}ek-Koman construction. They coincide when $ k $ divides $ n $ but are slightly different otherwise. They are widely believed to be asymptotically correct. In fact, the constructions in \cite{DaDASa} and \cite{DPS13} have the same number of crossings (which is in some cases smaller than that in \cite{SSSV96}), giving rise to the following conjecture on the $ k $-page book crossing number of $ K_n $ (as presented in \cite{DPS13}).
\begin{conj}\label{conj:k-page}
For any positive integers $ k $ and $ n $, $$ \nu_k (K_n)=Z_k(n).$$
\end{conj}
\'Abrego et al. \cite{AAFRS2} proved  Conjecture \ref{conj:k-page} for $ k=2 $, which can be rewritten as 

\[
\nu_2 (K_n)=Z_2(n)=\frac{1}{4}  \left\lfloor\frac{\mathstrut n}{\mathstrut 2}\right\rfloor  \left\lfloor\frac{n-1}{2}\right\rfloor  \left\lfloor\frac{n-2}{2}\right\rfloor  \left\lfloor\frac{n-3}{2}\right\rfloor.
\] 

The only other previously known exact values of $ \nu_k(K_n) $ are $ \nu_k (K_n)=0 $ for $ k\geq\ceil {n/2} $, as $ Z_k(n)=0 $ in this case \cite{DaDASa,DPS13,SSSV96}, and the remaining unshaded cells in Table \ref{table:known_values} (for which the conjecture holds, see \cite{DPS13}). In Section \ref{sect:exact}, we prove the conjecture for an infinite family of values, namely for any $ k $ and $ n $ such that $ 2< n/k \leq 3$ (shaded cells in Table \ref{table:known_values}, see Theorem \ref{theorem:range2to3}), and give improved lower bounds for $ n/k > 3$ (see Theorem \ref{theorem:rangeall}).

\begin{table}[h]
\begin{center}
\scalebox{0.7}{
\begin{tabular}{ | c | c| c | c| c | c| c| c | c| c| c| c| c| c| c| c| c| c|c|c|c|}
   \hline 
  $ n $ & 5 & 6 & 7 & 8 & 9 & 10 & 11 & 12 & 13& 14 & 15 & 16 & 17 & 18 & 19 & 20 & 21 & 22& $ \cdots  $& $ n $\\
   \hline 
     $\nu_2(K_n) $    & 1 & 3 & 9 & 18 & 36 & 60 & 100 & 150 & 225 & 315 & 441  & 588 & 784 & 1008 & 1296 & 1620 & 2025 & 2475 & & $ Z_2(n) $    \\
   \hline 
     $\nu_3(K_n) $ & 0 & 0 & 2 & 5 & 9 & 20 & 34 & 51 & 83 & 121 & 165 & - & - & - & - & - & - & -   &  & -\\
   \hline 
     $\nu_4(K_n) $ & 0 & 0 &  0 & 0 & 3 & 7 & 12 & 18 & 34 & - & - & - & - & - & - & - & -  & - &  & -\\
   \hline 
     $\nu_5(K_n) $ & 0 & 0 &  0 & 0 & 0 & 0 & 4 & 9 & \cellcolor[gray]{0.8}15 & \cellcolor[gray]{0.8}22 & \cellcolor[gray]{0.8}30& - & - & - & - & - & - & - &  & -\\
       \hline 
          $\nu_6(K_n) $ & 0 & 0 &  0 & 0 & 0 & 0 & 0 & 0 & \cellcolor[gray]{0.8}5 & \cellcolor[gray]{0.8}11 & \cellcolor[gray]{0.8}18 & \cellcolor[gray]{0.8}26 & \cellcolor[gray]{0.8}35 & \cellcolor[gray]{0.8}45 & - & - & -& - &  & -\\
       \hline 
          $\nu_7(K_n) $ & 0 & 0 &  0 & 0 & 0 & 0 & 0 & 0 & 0 & 0 & \cellcolor[gray]{0.8}6 & \cellcolor[gray]{0.8}13 & \cellcolor[gray]{0.8}21 & \cellcolor[gray]{0.8}30 & \cellcolor[gray]{0.8}40 & \cellcolor[gray]{0.8}51 & \cellcolor[gray]{0.8}63& - &  & -\\
       \hline 
       \vdots& & & & & & & & & & & & & & \vdots& & & & & & \\
       \hline 
       $\nu_k(K_n) $ & 0 & 0 &  0 & 0 & 0 & 0 & 0 & $ \cdots $ & 0 & 0 & 0 & 0  &\multicolumn{7}{|c|}{\cellcolor[gray]{0.8}$ \nu_k(K_n)=Z_k(n) $ for $ 2k<n\leq 3k $} & -\\
              \hline 
\end{tabular}
}
\end{center}
\caption{Known values of $ \nu_k(K_n) $. New values in this paper are shaded. (Cells corresponding to still unknown values of $ \nu_k(K_n)$ are marked with ``-''.)}
\label{table:known_values}
\end{table}

In terms of general lower bounds, Shahrokhi et al. \cite{SSSV96} proved a bound for $ \nu_k(G) $ for any graph $ G $. Using this bound for $ K_n $ gives
\[
\nu_k(K_n)\geq \frac{n(n-1)^3}{296k^2}-\frac{27kn}{37}=\frac{3}{37k^2}\binom{n}{4}+O(n^3).
\]
This general bound was improved by De Klerk et al. \cite{DPS13} to
\begin{equation}\label{eq:previous_lower_bound}
\nu_k (K_n) \geq \left\{
     \begin{array}{ll}
       \frac{3}{119}\binom{n}{4}+O(n^3) & \textup{if } k=4,\\
       \frac{2}{(3k-2)^2}\binom{n}{4}& \textup{if } k \textup{ is even, and } n\geq k^2/2+3k-1,\\
       \frac{2}{(3k+1)^2}\binom{n}{4}& \textup{if } k \textup{ is odd, and } n\geq k^2+2k-7/2.
     \end{array}
   \right.
\end{equation}

Using semidefinite programming, De Klerk et al. \cite{DPS13} further improved the lower bound for several values of $ k\leq 20 $. In Section \ref{sect:asympt}, we prove Theorem \ref{theorem:asympt_improvement} that improves the lower bounds in \cite{DPS13} for $ k\geq 14 $  (see Table \ref{table:improvement}), but more importantly, it improves the asymptotic bound (\ref{eq:previous_lower_bound}) for every $ k $. 

\begin{restatable}{thm}{LowerBookBound}\label{theorem:asympt_improvement}
For any integers $ k\geq 3 $ and $ n\geq \floor {111k/20} $,
\begin{multline*}
 \nu_k(K_n)\geq \frac{8000(4107k^2-5416k+1309)}{37(111k-17)(111k-77)(37k-19)(3k-1)}\binom{n}{4}
 =\left(\frac{8000}{12321}\cdot\frac{1}{k^2}+\Theta\left(\frac{1}{k^3}\right)\right)\binom{n}{4}. 
\end{multline*}
\end{restatable}
In contrast, $ Z_k(n) $ was asymptotically estimated in \cite{DPS13}, 
\[
\nu_k (K_n)\leq Z_k(n)=\left(\left(\frac{2}{k^2}\right)\left(1-\frac{1}{2k}\right)\right)\binom{n}{4}+O(n^3).
\]
This improves the ratio of the lower to the upper bound on $ \lim_{n\to \infty}\frac{\nu_k(K_n)}{\binom{n}{4}} $ from approximately $ \frac{1}{9} \approx 0.1111$ to $ \frac{4000}{12321} \approx 0.3246$.
All our results (exact values and asymptotic bounds) heavily rely on a different problem for convex graphs that is interesting on its own right. Before stating this new problem and our results, we describe its connection to book drawings.

\begin{table}[h]
\begin{center}
\scalebox{0.8}{
\begin{tabular}{ |D|C|C|C|C|}
   \hline 
  $ k $ & Lower bound in \cite{DPS13} & New lower bound & Upper bound & \parbox[t]{3.2cm}{Ratio of new lower\\ to upper bound}  \\
    \hline 
     $14 $ & $ 3.2930\times 10^{-3} $ & $ 3.4342\times 10^{-3} $ & $ 9.8396\times 10^{-3} $ & $ 0.3490 $ \\
   \hline 
     $15 $ & $ 2.5870\times 10^{-3} $ & $ 2.9852\times 10^{-3} $ & $ 8.5925\times 10^{-3} $ & $ 0.3474 $ \\
   \hline 
     $16 $ & $ 2.0348\times 10^{-3} $ & $ 2.6193\times 10^{-3} $ & $ 7.5683\times 10^{-3} $ & $ 0.3461 $  \\
   \hline 
     $17 $ & $ 1.6023\times 10^{-3} $ & $ 2.3166\times 10^{-3} $ & $ 6.7168\times 10^{-3} $ & $ 0.3449 $ \\
        \hline 
     $18 $ & $ 1.2562\times 10^{-3} $ & $ 2.0621\times 10^{-3} $ & $ 6.0013\times 10^{-3} $ & $ 0.3436 $ \\
        \hline 
     $19 $ & $ 9.8258\times 10^{-4} $ & $ 1.8490\times 10^{-3} $ & $ 5.3943\times 10^{-3} $ & $ 0.3428 $ \\
        \hline
     $20 $ & $ 7.7482\times 10^{-4} $ & $ 1.6653\times 10^{-3} $ & $ 4.8750\times 10^{-3} $ & $ 0.3416 $ \\
        \hline 
\end{tabular}
}
\end{center}
\caption{ Bound comparison for $ \displaystyle\lim _{n\to \infty}\frac{\nu_k(K_n)}{\binom{n}{4}}. $}
\label{table:improvement}
\end{table}

One way to visualize and study $ k $-page book drawings is by using what we called the \emph{convex model}. In this model, a given $ k $-page book drawing of a graph $ G $ is drawn on the plane as an edge-colored \emph{convex drawing} as follows: The spine is now a circle $ C $, or more generally, a simple convex curve. The vertices of $ G $ are placed on $ C $, typically forming the set of vertices of a convex polygon inscribed in $ C $. The edges are diagonals or sides (straight line segments) of the polygon that are $ k $-colored in such a way that two edges get the same color if and only if they originally were on the same page. Using this model, the problem of determining $ \nu_k(K_n) $ is equivalent to finding the minimum number of monochromatic crossings in a $ k $-edge coloring of a convex drawing of $ K_n $. The subgraph (subdrawing) induced by each of the colors is known as a \emph{convex}  or \emph{outerplanar graph} (drawing). We denote by $ G_n $ the complete convex graph, that is, any convex drawing of $ K_n $. Since we are interested in crossings, it is often convenient to disregard the sides of the underlying polygon as edges. We denote by $ D_n $ the complete convex graph minus all the edges corresponding to the sides of the underlying polygon. Let $e_\ell(n)$ be the maximum number of edges over all convex subgraphs of $ D_n $ in which each edge is crossed at most $ \ell $ times. Such graphs are said to have \emph{local crossing number} at most $ \ell $. For example,  planar graphs have local crossing number 0, the graphs in Figure \ref{fig:constructions_all_e}(b) have local crossing number 1, and those in \ref{fig:constructions_all_e}(e) have local crossing number 4. (Graphs with local crossing number at most $ \ell $ are called \emph{$ \ell $-planar graphs} \cite{R65}. If they are also convex, then they are called \emph{outer $ \ell $-planar graphs} \cite{E86,Kai90,ABB16}.) Local crossing numbers of convex graphs were studied by Kainen \cite{Kai73,Kai90}. The problem of maximizing the number of edges over convex graphs satisfying certain crossing conditions was studied by Brass, K\'arolyi, and Valtr \cite{BKV03}. 
Functions equivalent to $ e_\ell(n) $ for general drawings of graphs in the plane were studied by Pach and T\'oth \cite{PT97}; Pach, Radoi\u{c}i\'c, Tardos, and T\'oth \cite{PRTT06}; and Ackerman \cite{A15}. 

In Section \ref{sect:theoremL}, we prove the following theorem that relates the functions $ e_\ell(n) $ to the $k$-page book crossing numbers $\nu_k(K_n) $. Theorem \ref{theorem:RangeGeneralIntro} is used to prove Conjecture \ref{conj:k-page} for $2k<n\leq3k$, the asymptotic bound of Theorem \ref{theorem:asympt_improvement}, and the lower bound improvements in Table \ref{table:improvement}.

\begin{thm}\label{theorem:RangeGeneralIntro}
	Let $ n \geq 3$ and $ k \geq 3 $ be fixed integers. Then, for all integers $ m\geq 0$,
	\[
	\nu_{k}(K_n) \geq \frac{m}{2}n(n-3)-k\sum_{\ell=0}^{m-1}e_\ell(n).
	\]
\end{thm}

In Section \ref{sect:maxedges}, we solve the problem of finding the maximum number of edges among $\ell$-planar convex graphs with no sides when $0 \le \ell \le 4$, in other words we obtain the exact value of the function $ e_\ell(n) $ for $0 \le \ell \le 4$. In Section \ref{sect:bound_e}, we present constructions of graphs with many edges and prescribed local crossing number $\ell$. We also describe the asymptotic behavior of $e_\ell (n)$. In Section \ref{sect:exact_e} we prove that these constructions are optimal by proving the corresponding inequalities required to settle the theorem that follows. 
\begin{restatable}{thm}{ExactValuesE}\label{theorem:epsilons}
	For any $ n\geq 3 $, with the exception of $e_4(3)=0$ and $e_4(4)=2$, 
 \begin{align}
     e_0(n)&=n-3\\
     e_1(n)&=\frac{3}{2}(n-3)+
     \delta_1(n)\\
    e_2(n)&=2(n-3)+
     \delta_2(n)\\
    e_3(n)&=\frac{9}{4}(n-3)+
     \delta_3(n)\\
    e_4(n)&=\frac{5}{2}(n-3)+
     \delta_4(n)
 \end{align}\vspace{-.2in}
 where 
\begin{table}[h!]
	\begin{center}
	\scalebox{.9}
	{{\renewcommand{\arraystretch}{1.5}
			\begin{tabular}{c|c|c|c}
				$\delta_1(n)$ & $\delta_2(n)$ & $\delta_3(n)$ & $\delta_4(n)$ \\
				\hline
					$\begin{array}{ll}
					1/2 & \textup{if } n \equiv 0 \pmod 2,\\
					0   & \textup{otherwise.}
					\end{array}$
					&
					$\begin{array}{ll}
					1 & \textup{if } n \equiv 2 \pmod 3,\\
					0 & \textup{otherwise.}
					\end{array}$
					& 
					$\begin{array}{ll}
					-1/4 & \textup{if } n \equiv 0 \pmod 4,\\
					1/2  & \textup{if } n \equiv 1 \pmod 4,\\
					5/4  & \textup{if } n \equiv 2 \pmod 4,\\
					0    & \textup{if } n \equiv 3 \pmod 4.
					\end{array}$
                    & 
					$\begin{array}{ll}
					1/2 & \textup{if } n \equiv 0 \pmod 4,\\ 
					0   & \textup{if } n \equiv 1 \pmod 4,\\
					3/2 & \textup{if } n \equiv 2 \pmod 4,\\
					1   & \textup{if } n \equiv 3 \pmod 4.
					\end{array}$		
			\end{tabular}}}\vspace{-.3in}
	\end{center}
\end{table}
\end{restatable}
\section{$ k $-page book crossing number of $ K_n $}\label{sect:kbook}
\subsection{Proof of Theorem \ref{theorem:RangeGeneralIntro}}\label{sect:theoremL}
In this section, we prove the lower bound on $ \nu_k(K_n) $ stated in Theorem \ref{theorem:RangeGeneralIntro}. This bound depends on the values of $ e_\ell (n) $. We first give a lower bound on the number of crossings of any convex graph $ H $ in terms of the functions $e_0 (n), e_1 (n), e_2 (n), e_3 (n), e_4(n),\dots$ Recall that $ D_n $ is the complete convex graph without the sides of the underlying polygon. In what follows, for a convex graph $H$ (a drawing), we denote by  $e(H)$, $\crg (H)$, and $\crg_H(a)$ the number of edges in $H$, the number of crossings in $H$, and the number of crossings with the edge $a$ in $H$, respectively.

\begin{thm}\label{theorem:CrossingsMaxEdges}
	Let $ n\geq 3 $ and $ m\geq 0 $ be any  integers. If $ H $ is any subgraph of $ D_n $, then 
	\begin{equation}\label{eq:CrossingsMaxEdges}
	\crg (H) \geq me(H)-\sum_{\ell=0}^{m-1}e_\ell(n).
	\end{equation}
\end{thm}

\begin{proof}
	We prove the result by induction on $ m $. Inequality (\ref{eq:CrossingsMaxEdges}) is trivially true for $ m=0 $ as the right hand side of (\ref{eq:CrossingsMaxEdges}) is 0. Let $ m\geq 1 $ and let $ H $ be a subgraph of $ D_n $. We consider two cases. 
	
	\paragraph{Case 1:} $ e(H)\leq e_{m-1}(n) $. By induction 
	\begin{eqnarray}
	\crg (H) &\geq &(m-1)e(H)-\sum_{\ell=0}^{m-2}e_\ell(n)\nonumber\\&=& (m-1)e(H)-\sum_{\ell=0}^{m-2}e_\ell(n)+e(H)-e(H)+e_{m-1}(n)-e_{m-1}(n) \nonumber \\
	& = & me(H)-\sum_{\ell=0}^{m-1}e_\ell(n)+(e_{m-1}(n)-e(H)) \geq me(H)-\sum_{\ell=0}^{m-1}e_\ell(n).   \nonumber 
	\end{eqnarray}
	
	\paragraph{Case 2:} $ e(H) > e_{m-1}(n) $. Let $ c= e(H) - e_{m-1}(n)$ and $ H_0=H $. For $ 0\leq i \leq c-1 $, recursively
	define $ a_i $ and $ H_{i+1} $ as follows. The graph $ H_i $ has $ e(H)-i\geq e(H)-(c-1)>e_{m-1} (n)$ edges. Thus, by definition of $ e_{m-1}(n) $, there exists an edge $ a_i $ in $ H_i $ that is crossed at least $ m $ times. Let $ H_{i+1} $ be the graph obtained from $ H_i $ by deleting $ a_i $. We have the following. 
	\begin{equation*}
	\crg (H)  =  \sum_{i=0}^{c-1}\crg{\hspace{-0.2em}}_{H_i}(a_i) +\crg (H_c) \geq mc+\crg (H_c)
	\end{equation*}
	By induction, $ \crg (H_c)\geq (m-1)e(H_c)-\sum_{\ell=0}^{m-2}e_\ell(n)$ and thus
	\begin{eqnarray}
	\crg (H) & \geq & mc+ (m-1)(e(H)-c)-\sum_{\ell=0}^{m-2}e_\ell(n)\nonumber\\
	&=& m(e(H)-e_{m-1}(n))+(m-1)e_{m-1}(n)-\sum_{\ell=0}^{m-2}e_\ell(n)=me(H)-\sum_{\ell=0}^{m-1}e_\ell(n).\nonumber
	\end{eqnarray}
\end{proof}

For any integers $ k\geq 1 $, $ n\geq 3 $, and $ m \geq 0 $, define
\begin{equation}\label{eq:defLk,n(m)}
L_{k,n}(m):=\frac{m}{2}n(n-3)-k\sum_{\ell=0}^{m-1}e_\ell(n).
\end{equation}
Note that $L_{k,m}$ is precisely the right hand side of the inequality in Theorem \ref{theorem:RangeGeneralIntro}. Then Theorem \ref{theorem:RangeGeneralIntro} can be stated as follows and we prove it using Theorem \ref{theorem:CrossingsMaxEdges}.

\begin{thm}[Equivalent to Theorem \ref{theorem:RangeGeneralIntro}]\label{theorem:RangeGeneral}
	Let $ n \geq 3$ and $ k \geq 3 $ be fixed integers. Then, for all integers $ m\geq 0$,
	\[
	\nu_{k}(K_n) \geq L_{k,n}(m).
	\]
\end{thm}

\begin{proof}
	Consider any $ k $-coloring of the edges of $ D_n $ using colors $ 1,2,\ldots,k $. Let $ H_i $ be the graph on the same set of vertices as $ D_n $ whose edges are those of color $ i $.  By Theorem \ref{theorem:CrossingsMaxEdges},
	\begin{equation}\label{eq:twostepinequality}
	cr(H_i)\geq me(H_i)-\sum_{\ell=0}^{m-1}e_\ell(n).
	\end{equation}
	Add (\ref{eq:twostepinequality}) over all colors, $ 1\leq i \leq k $, to show that the number of monochromatic crossings in the $ k $-edge coloring is at least 
	\begin{eqnarray}
	\sum_{i=1}^{k} \left(me(H_i) -\sum_{\ell=0}^{m-1}e_\ell(n)\right) &=&
	me(D_n)- k\sum_{\ell=0}^{m-1}e_\ell(n)\nonumber\\
	=m\left({n \choose 2} - n\right) -k\sum_{\ell=0}^{m-1}e_\ell(n)
	&=&\frac{m}{2}n(n-3)-k\sum_{\ell=0}^{m-1}e_\ell(n)=L_{k,n}(m).\nonumber 
	\end{eqnarray}
\end{proof}

\subsection{Exact values of $ \nu_k(K_n) $ and new lower bounds}\label{sect:exact}
Since Theorem \ref{theorem:RangeGeneral} works for any nonnegative integer $ m $, we start by maximizing $ L_{k,n}(m) $ as defined in (\ref{eq:defLk,n(m)}) for fixed $ k $ and $ n $ to obtain the best possible lower bound provided by Theorem \ref{theorem:RangeGeneral}.

\begin{prop}\label{prop:RangeGeneral}
	For fixed integers $ k\geq 1 $ and $ n\geq 3 $, the value of  $ L_{k,n}(m) $, defined over all nonnegative integers $ m $, is maximized by the smallest $ m $ such that $   e_{m}(n) \geq \frac{n(n-3)}{2k}$.
\end{prop}

\begin{proof}
	Note that $$ L_{k,n}(m)=\sum_{\ell=0}^{m-1} \left(\frac{n}{2}(n-3)-ke_\ell(n)\right)
	=k\sum_{\ell=0}^{m-1} \left(\frac{n(n-3)}{2k}-e_\ell(n)\right). $$ So $ L_{k,n} $ is increasing as long as $ e_{m-1}(n)<\frac{n(n-3)}{2k} $ and nonincreasing afterwards. Thus the maximum is achieved when $ m $ is the smallest integer such that $e_m(n) \geq \frac{n(n-3)}{2k}$. Note that this value in fact exists because $ e_\ell(n)=\binom{n}{2}-n\geq \frac{n(n-3)}{2k} $ whenever $ \ell > \floor{\frac{n-2}{2}}\ceil{\frac{n-2}{2}} $. 
\end{proof}

We now use Proposition \ref{prop:RangeGeneral} to explicitly state the best lower bounds guaranteed by Theorem \ref{theorem:RangeGeneral} using the values of $ e_\ell(n) $ obtained in Section \ref{sect:maxedges}. In what follows, $ \delta_1(n) $, $ \delta_2(n) $, $ \delta_3(n) $, and $ \delta_4(n) $ are defined as in Theorem \ref{theorem:epsilons}.
\begin{thm}\label{theorem:rangeall}
	For any integers $ k\geq 3 $ and $ n> 2k $,
	\begin{equation}\label{ineq:rangeall}
	\nu_{k}(K_n) \geq 
	\left\{
	\begin{array}{ll}
	\mathstrut L_{k,n}(1)=\frac{1}{2}(n-3)(n-2k) & \textup{if } 2k <n\leq 3k,\vspace{0.05in}\\
	L_{k,n}(2)=(n-3)(n-\frac{5}{2}k)-k\delta_1(n)                 & \textup{if } 3k <n\leq 4k,\\
	L_{k,n}(3)=\frac{3}{2}(n-3)(n-3k)-k(\delta_1+\delta_2)(n) & \textup{if } 4k <n\leq 	\floor{4.5k} +\beta\\
	L_{k,n}(4)=2(n-3)(n-\frac{27}{8}k)-k(\delta_1+\delta_2+\delta_3)(n)  & \textup{if } \floor{4.5k} +\beta <n\leq 5k\\
	L_{k,n}(5)=\frac{5}{2}(n-3)(n-\frac{37}{10}k)-k(\delta_1+\delta_2+\delta_3+\delta_4)(n)  & \textup{otherwise.}
	\end{array}
	\right.
	\end{equation}
	where $ \beta=
	\left\{
	\begin{array}{ll}
	-1 & \hspace{-.1in}\textup{if } k \textup{ even and } 4|n,\\
	1 & \hspace{-.1in}\textup{if } k \textup{ odd and } 4|n-2,\\
	0  & \hspace{-.1in}\textup{otherwise,}
	\end{array}
	\right. $
\end{thm}

\begin{proof}
	Let $1\leq m\leq 5$. The $ m ^{th}$ row on the right hand side of Inequality (\ref{ineq:rangeall}) is equal to $ L_{k,n}(m) $, using the values of $ e_0(n) $,  $ e_1(n) $, $ e_2(n) $, $ e_3(n)$, and $ e_4(n) $, stated in Theorem \ref{theorem:epsilons}. In each case, the range for $ n $ corresponds to the values of $ n $ for which $ e_{m-1}(n)<\frac{n(n-3)}{2k} \leq e_m(n)$, thus guaranteeing the best possible bound obtained from Theorem \ref{theorem:RangeGeneral} by Proposition \ref{prop:RangeGeneral}. For example, by Theorems \ref{theorem:epsilons} and \ref{theorem:RangeGeneral}, 
	\[
	\nu_k(K_n) \geq L_{k,n}(1)=\frac{1}{2}n(n-3)-ke_0(n)= \frac{1}{2}n(n-3)-k(n-3)=\frac{1}{2}(n-3)(n-2k).
	\]
	Although this bound holds for any $ n $ and is tight for $2 < \frac{n}{k} \leq 3$ (as stated in the next result), it can actually be improved for larger values of $ n $. For instance, using $ m=2 $
	in Theorem \ref{theorem:RangeGeneral} together with Theorem \ref{theorem:epsilons} yields
	\[
	\nu_k(K_n) \geq L_{k,n}(2)=n(n-3)-k\left((n-3)+\frac{3}{2}(n-3)+\delta_1(n)\right)=(n-3)(n-\frac{5}{2}k)-\delta_1(n)k.
	\]
	By Proposition \ref{prop:RangeGeneral}, this is the best bound guaranteed by Theorem \ref{theorem:RangeGeneral} whenever $ 3<n/k\leq 4 $ because $ e_1(n)=\frac{3}{2}(n-3)+\delta_1(n)<\frac{n(n-3)}{2k} \leq 2(n-3)\leq e_2(n)$ in this range.
\end{proof}

The first part of Theorem \ref{theorem:rangeall} settles Conjecture \ref{conj:k-page} when $2 < \frac{n}{k} \leq 3$.

\begin{thm}\label{theorem:range2to3}
	If $2 < \frac{n}{k} \leq 3$, then $$\nu_{k}(K_n) = \frac{1}{2}(n-3)(n-2k).$$
\end{thm}

\begin{proof}
	Suppose $2 < \frac{n}{k} \leq 3$. Since $\nu_{k}(K_n)\leq Z_k(n) $, we just need to show that $ Z_k(n) $ is indeed equal to $ \frac{1}{2}(n-3)(n-2k) $ for these values of $ n $ and $ k $. By (\ref{eq:zk2}), $ F(2,n)=0 $ and $ F(3,n)=\frac{1}{2}(n-3) $. If $ n=3k $, then by (\ref{eq:zk1}),
	\[ 
	Z_k(n)= k\cdot F (3,n)=k\cdot\frac{1}{2}(n-3)=\frac{1}{2}(n-3)(n-2k).
	\]
	If $2 < \dfrac{n}{k} < 3$, then $ n=2k+r $, where $  r=n \bmod k $. By (\ref{eq:zk1}),
	\[ 
	Z_k(n)= r\cdot F (3,n)+(k-r) \cdot F (2,n)=\frac{1}{2}(n-3)(n-2k).
	\]
\end{proof}

\subsection{Improving the asymptotics for fixed $ k $}\label{sect:asympt}

The bound in Theorem \ref{theorem:rangeall} becomes weaker as $ n/k $ grows. We now use a different approach to improve this bound when $ n $ is large with respect to $ k $. 

\LowerBookBound*

\begin{proof}
For fixed $ k $ and all $ n\geq n'\geq 4 $, it is known that  (see for example \cite{RT97} or Theorem 2 in \cite{SW94})
$$ \frac{\nu_k(K_n)}{\binom{n}{4}} \geq \frac{\nu_k(K_{n'})}{\binom{n'}{4}}. $$
Let $n'=\lfloor 111k/20 \rfloor$ and note that $n'>5k$ for $k\ge 3$. Thus by Theorem \ref{theorem:rangeall} it follows that 
\begin{equation}\label{eq:asymptotics}
\frac{\nu_k(K_n)}{\binom{n}{4}} \geq \frac{\nu_k(K_{n'})}{\binom{n'}{4}}\ge \frac{\tfrac{5}{2}(n'-3)(n-\tfrac{37}{10}k)-k(\delta_1+\delta_2+\delta_3+\delta_4)(n')}{{\binom{n'}{4}}}. 
\end{equation}
The right-hand side of (\ref{eq:asymptotics}) can be simplified by considering the possible residue classes of $k$ modulo 80. In particular, the minimum of all such expressions over all the residue classes modulo 80 occurs when $k \equiv 7 \pmod {80}$. This is how we obtain a universal bound that does not depend on the residue class. The simplified expression for  $k \equiv 7 \pmod {80}$ gives
\[
\frac{\nu_k(K_n)}{\binom{n}{4}} \ge \frac{8000(4107k^2-5416k+1309)}{37(111k-17)(111k-77)(37k-19)(3k-1)}.
\]
Finally, using $ n'=\floor{111k/20}$ for $ 14\leq k \leq 20$ in (\ref{eq:asymptotics}), we obtain the bounds in Table \ref{table:n_15_to_20}. These new bounds improve those in \cite{DPS13} as shown in Table \ref{table:improvement}.
\end{proof}

\begin{table}[h]
	\begin{center}
		\scalebox{1}{
			\begin{tabular}{  DEE |DEE}
				$ k \vspace{.1in}$ & $ \nu_k(K_n) \geq $ \vspace{.1in}& for all $ n\geq $ \vspace{.1in} &   $ k \vspace{.1in}$ & $ \nu_k(K_n) \geq $ \vspace{.1in}& for all $ n\geq $ \vspace{.1in}\\
				\hline 
				\vspace{.1in}
				$ 14 $ & $\frac{4406}{1282975}\binom{n}{4} $ & $ 76 $ &  $ 18 $ & $\frac{8086}{3921225}\binom{n}{4}  $ &  $ 100 $ \\
				\vspace{.1in}
				$ 15 $ & $\frac{640}{214389}\binom{n}{4} $ & $ 84 $ &  $ 19 $ & $\frac{8839}{4780230}\binom{n}{4}  $ &  $ 105 $ \\
				\vspace{.1in}
				$ 16 $ & $ \frac{3054}{1165945}\binom{n}{4} $ & $ 88 $ &   $ 20 $ & $ \frac{85}{51039}\binom{n}{4}  $ & $ 109 $\\
				\vspace{.1in}
				$ 17 $ & $ \frac{6764}{2919735}\binom{n}{4}$ & $ 93 $ & &  &  \\
			\end{tabular}
		}
	\end{center}
	\caption{New lower bounds for $ \nu_k(K_n)$ when $ 14\leq k \leq 20 $.}
	\label{table:n_15_to_20}
\end{table}

\section{Maximizing the number of edges}\label{sect:maxedges}

In this section, we only consider convex graphs that do not use the sides of the polygon, that is, subgraphs of $ D_n $. Given such a graph $ G $ and two of its vertices $ x $ and $ y $, the segment $ xy $ is either an \emph{edge} of $ G $ (which must be a diagonal of the polygon), a \emph{side} of $ G $ (a side of the polygon, and thus not an edge of $ G $), or a \emph{nonedge} of $ G $ (a diagonal of the polygon that is not an edge of $ G $). Let $ e(G) $, $ cr(G) $, and $ \lc(G) $ denote the number of edges, number of crossings, and local crossing number of $ G $, respectively. 

As defined before, $e_\ell(n)$ denotes the maximum number of edges over all subgraphs of $ D_n $ with local crossing number at most $ \ell $. 
In Section \ref{sect:bound_e}, we show constructions that provide lower bounds for the function $e_{\ell}(n)$ for every $0 \le \ell \le 4$ and in Section \ref{sect:exact_e}, we determine the exact value of $e_{\ell}(n)$ for $\ell \leq 4$ by proving the corresponding upper bounds.

\subsection{Lower bounds for $e_{\ell}(n)$: constructions}\label{sect:bound_e}

For every $0 \le \ell \le 4$, and $n\ge 3$, we construct convex graphs with many edges and where every edge is crossed at most $\ell$ times. 

\begin{thm} \label{thm:lower bound}For any $0 \le \ell \le 4$ and $n\ge 3$, there are convex subdrawings $G_{\ell,n}$ of $D_n$ such that $\lc(G_{\ell,n}) \le \ell$ and
\begin{align*}
     e_0(G_{0,n})&=n-3\\
     e_1(G_{1,n})&=\frac{3}{2}(n-3)+\delta_1(n)\\
    e_2(G_{2,n})&=2(n-3)+\delta_2(n)\\
    e_3(G_{3,n})&=\frac{9}{4}(n-3)+\delta_3(n)\\
    e_4(G_{4,n})&=\frac{5}{2}(n-3)+\delta_4(n)\text{ for }n\ge 5.
 \end{align*}
\end{thm}

\begin{proof}
    The constructions are simple. For every $0 \le \ell \le 4$, we begin with an appropriate graph $B_\ell$ which is the largest complete graph (or close to it) that has local crossing number at most $\ell$. In fact, $B_0$, $B_1$, $B_2$, $B_3$, and $B_4$ are respectively, $D_3$, $D_4$, $D_5$, $D_6^-$ (defined as $D_6$ minus one of its main diagonals), and $D_6$. The idea is to join together several copies of $B_\ell$, using uncrossed diagonals, to an appropriate initial graph with local crossing number at most $\ell$ (see Figure \ref{fig:constructions_all_e}). The specific details follow.
    If $\ell=0$, then $G_{0,n}$ consists of joining together $n-2$ copies of $D_3$, see Figure \ref{fig:constructions_all_e}(a). That is, $G_{0,n}$ consists of the inner edges of a triangulation of a convex graph on $n$ vertices and so $e(G_{0,n})=n-3$.
    

     If $\ell=1$, then $G_{1,n}$ consists of joining $\lfloor (n-3)/2 \rfloor$ copies of $D_4$ and a single copy of an initial graph depending on the parity of $n$, see Figure \ref{fig:constructions_all_e}(b).   If $n$ is even, then the initial graph is $D_4$; and if $n$ is odd, then the initial graph is $D_3$. The resulting convex graph has $n$ vertices and local crossing number at most 1. There are 2 edges for every $D_4$, either 0 or 2 edges in the initial graph, and $\lfloor (n-3)/2 \rfloor$ uncrossed diagonals for a total number of edges  
     \begin{align*}
     e(G_{1,n})&=2 \left\lfloor \frac{n-3}{2} \right\rfloor+\left\lfloor \frac{n-3}{2} \right\rfloor+\begin{cases}
		2 & \textup{if } n \textup{ is odd},\\
	    0  &  \textup{if } n \textup{ is even.} \end{cases}\\
     &= \tfrac{3}{2}(n-3)+\delta_1(n).
     \end{align*}
 
     If $\ell=2$, then $G_{2,n}$ consists of joining together $\lfloor (n-3)/3 \rfloor$ copies of $D_5$ and a single copy of one of the following initial graphs depending on $n\pmod 3$: $D_3$ if $n \equiv 0 \pmod 3$, $D_4$ if $n \equiv 1 \pmod 3$, and $D_5$ if $n \equiv 2 \pmod 3$; see Figure \ref{fig:constructions_all_e}(c). The resulting convex graph has $n$ vertices and local crossing number at most 2. There are 5 edges for every $D_5$, either $0$, $2$, or $5$ edges in the initial graph, and $\lfloor (n-3)/3 \rfloor$ uncrossed diagonals for a total number of edges   
     \begin{align*}
     e(G_{2,n})&=5 \left\lfloor \frac{n-3}{3} \right\rfloor+\left\lfloor \frac{n-3}{3} \right\rfloor+\begin{cases}
		0 & \textup{if }n \equiv 0 \pmod 3,\\
	    2  &  \textup{if } n \equiv 1 \pmod 3,\\
        5   &  \textup{if } n \equiv 2 \pmod 3 \end{cases}\\
     &= 2(n-3)+\delta_2(n).
     \end{align*}  
     
\begin{figure}[bt]
	\centering
	\includegraphics[width=.8\linewidth]{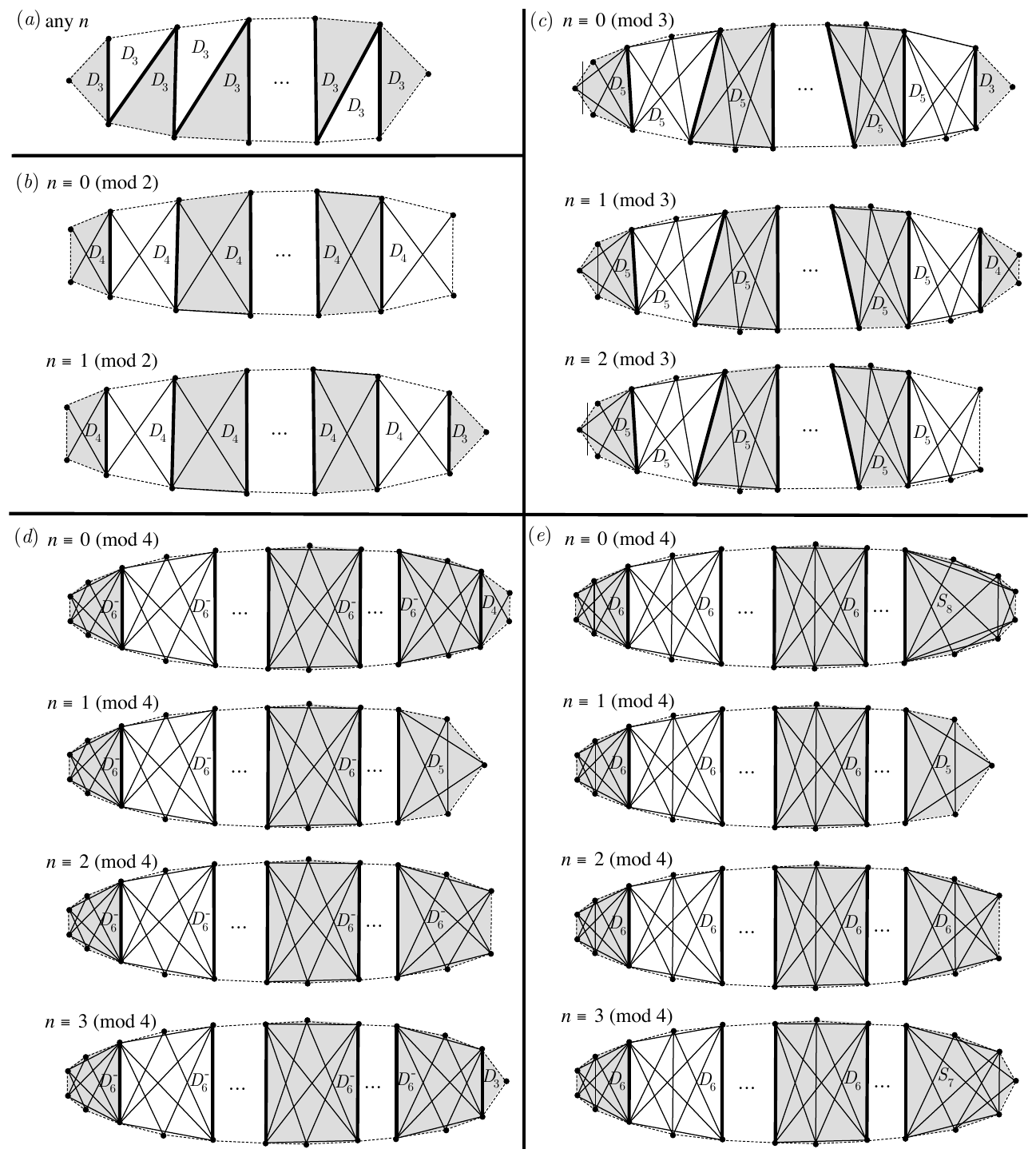}
	\caption{The graphs $G_{\ell,n}$, $0\leq \ell \leq 4$, which maximize the number of edges among those convex graphs with $n$ vertices and local crossing number $\ell$: (a) $\ell=0$, (b) $\ell=1$, (c) $\ell=2$, (d) $\ell=3$, and (e) $\ell=4$.}
	\label{fig:constructions_all_e}
\end{figure}
        
       If $\ell=3$, then $G_{3,n}$ consists of joining together $\lfloor (n-3)/4 \rfloor$ copies of $D_6^-$ and a single copy of one of the following initial graphs depending on $n\pmod 4$:  $D_4$ if $n \equiv 0 \pmod 4$, $D_5$ if $n \equiv 1 \pmod 4$,  $D_6^-$ if $n \equiv 2 \pmod 4$, and $D_3$ if $n \equiv 3 \pmod 4$; see Figure \ref{fig:constructions_all_e}(d). The resulting graph has $n$ vertices and local crossing number at most 3. There are 8 edges for every $D_6^-$, either $2$, $5$, $8$, or $0$ edges in the initial graph, and $\lfloor (n-3)/4 \rfloor$ uncrossed diagonals for a total number of edges  
     \begin{align*}
     e(G_{3,n})&=8 \left\lfloor \frac{n-3}{4} \right\rfloor+\left\lfloor \frac{n-3}{4} \right\rfloor+\begin{cases}
		2 & \textup{if }n \equiv 0 \pmod 4,\\
	    5  &  \textup{if } n \equiv 1 \pmod 4,\\
        8  &  \textup{if } n \equiv 2 \pmod 4,\\
        0   &  \textup{if } n \equiv 3 \pmod 4. \end{cases}\\
     &= \tfrac{9}{4}(n-3)+\delta_3(n).
     \end{align*}

If $\ell=4$ and $n\ge 5$, then $G_{4,n}$ consists of joining together $\lfloor (n-5)/4 \rfloor$ copies of $D_6$ and a single copy of one of the following initial graphs depending on $n\pmod 4$. If $n \equiv 0 \pmod 4$, then the initial graph is the graph $S_8$ in Figure \ref{fig:FigGraphsC4}(c), it has 8 vertices, 13 edges, and local crossing number 4, if $n \equiv 1 \pmod 4$, then the initial graph is $D_5$, if $n \equiv 2 \pmod 4$, then the initial graph is $D_6$,  and if $n \equiv 3 \pmod 4$, then the initial graph is either $S_7$ or $S'_7$ shown in Figure \ref{fig:FigGraphsC4}(a-b), it has 7 vertices, 11 edges, and local crossing number 4. The resulting graph has $n$ vertices and local crossing number at most 4. There are 9 edges for every $D_6$, either $13$, $5$, $9$, or $11$ edges in the initial graph, and $\lfloor (n-5)/4 \rfloor$ uncrossed diagonals for a total number of edges 
     \begin{align*}
     e(G_{4,n})&=9 \left\lfloor \frac{n-5}{4} \right\rfloor+\left\lfloor \frac{n-5}{4} \right\rfloor+\begin{cases}
		13 & \textup{if }n \equiv 0 \pmod 4,\\
	    5  &  \textup{if } n \equiv 1 \pmod 4,\\
        9  &  \textup{if } n \equiv 2 \pmod 4,\\
        11   &  \textup{if } n \equiv 3 \pmod 4. \end{cases}\\
     &= \tfrac{5}{2}(n-3)+\delta_4(n).
     \end{align*}    
\end{proof}

\begin{figure}
\centering
\includegraphics[width=0.7\linewidth]{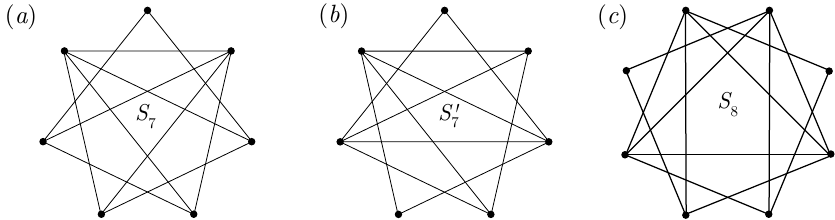}
\caption{(a-b) The graphs $ S_7 $ and $ S_7' $, the unique subgraphs of $ D_7 $ with $ 11 $ edges and local crossing number $ 4 $. (c) The graph $ S_8 $, a subgraph of $ D_8 $ with $ 13 $ edges and local crossing number $ 4 $. }
\label{fig:FigGraphsC4}
\end{figure}

Following similar ideas, it is possible to extend these constructions for larger values of $\ell$. These constructions would imply that $e_{\ell} (n) \ge \sqrt{\ell} n+\Theta(1)$. We avoid giving the details because we do not believe these constructions are optimal when $\ell \ge 7$. 

On the other hand, $ \crg (G)\geq {m^3}/({27n^2}) $ for any convex graph with $ n $ vertices and $ m \geq 3n$ edges \cite{SSSV04}. So if $ G $ is a subgraph of $ D_n $ such that $ \lc (G)=\ell $ and $ e(G)=e_\ell(n) $, then 
\begin{equation}\label{Ineq:CrossingLemma1}
\crg(G)\geq \frac{e_\ell(n)^3}{27n^2}.
\end{equation}
Also, because each edge of $ G $ is crossed at most $ \ell $ times, 
\begin{eqnarray}\label{Ineq:CrossingLemma2}
\crg(G)\leq \frac{e_\ell(n)\ell}{2}.
\end{eqnarray}
Inequalities (\ref{Ineq:CrossingLemma1}) and (\ref{Ineq:CrossingLemma2}) imply that
$
e_\ell(n) \leq \sqrt{\frac{27\ell}{2}}n<3.675\sqrt{\ell}n
$, and so in general
\[
\sqrt{\ell} n+\Theta(1) \le e_{\ell}(n) \le 3.675\sqrt{\ell}n.
\]

\subsection{Upper bounds for $e_{\ell}(n)$: settling the cases $\ell \leq 4$}\label{sect:exact_e}
In this section, we prove Theorem \ref{theorem:epsilons} (restated below), which settles the values of $e_{\ell}(n)$ for $\ell \leq 4$. Similar results for $ \ell=0 $ and 1 were proved in \cite{BKV03}. Theorem 3 in \cite{BKV03} proves that the maximum number of edges in a plane subgraph of $ G_n $ (instead of $ D_n $) is $ 2n-3 $. This is equivalent to $ e_0(n)=n-3 $ in Theorem \ref{theorem:epsilons}, as the sides of the polygon are included as possible edges, but their proof is different. Theorem 8 in \cite{BKV03}, however, is somewhat different from our result $  e_1(n)=\frac{3}{2}(n-3)+\delta_1(n) $. It shows that the largest number of edges in a subgraph of $ G_n $ such that any two edges sharing a vertex are not both crossed by a common edge is $\floor*{\frac{5}{2}n-4}=n+\ceil*{\frac{3}{2}(n-3)}$. Recall that this result uses the sides of the polygon.

\ExactValuesE*


\begin{proof}
    For convenience, we define $\delta_0(n)=0$ for all $n\geq 3$. For $0\leq \ell \leq 4$, we need to prove that $e_\ell(n)=C_\ell(n-3)+\delta_\ell(n)$, where $C_0=1, C_1=3/2, C_2=2, C_3=9/4$, and $C_5=5/2$. The individual proofs for the different values of $\ell$ considerably increase in difficulty as $\ell$ increases, but they are almost contained one in the next. Presenting these five proofs separately would be too repetitive and thus we decided to present them together using the parameter $\ell$. What this means is that early in the proof we conclude with the cases $\ell=0$ and 1; a bit further with cases $\ell = 2$ and $3$; and the rest of the proof is only for the case $\ell=4$.

     As before, we label the vertices of $ G $ by $ v_1,v_2,\cdots , v_n$ in counterclockwise order along the boundary of the underlying polygon of $ D_n $. Clearly, $e_{\ell}(3)=0$ for any nonnegative integer $\ell$, $e_0(4)=1$, and $e_{\ell}(4)=2$ for $1 \leq \ell\leq 4$. So assume that $ 0\leq \ell \leq 4 $ and $n\geq 5$. The inequality $ e_{\ell}(n) \geq C_{\ell}(n-3)+\delta_{\ell}(n) $ holds by Theorem \ref{thm:lower bound}. To assert equality, we prove the inequality $e_\ell(n) \leq C_\ell(n-3)+\delta_\ell(n)$ by induction on $ n $. 	This can be verified for $ 5\leq n \leq \ell+3 $. Assume $ n\geq\ell+4 $ and let $G$ be a subgraph of $D_n$ with $n$ vertices and local crossing number at most $ \ell $.
    
    Figure \ref{fig:organization} presents the general organization of the rest of the proof, here are some details. The rest of the proof is divided into three cases. Case 1  finishes the proof (using induction) when $G$ has at least one uncrossed edge. This always happens when $\ell=0$ or $1$ concluding those cases. When all edges of $G$ are crossed, the situation is more complicated. In order to understand the behavior of the crossings in this case, we define the \emph{crossing graph} of $ G $, denoted by $ G^{\otimes} $, as the graph whose vertices are the edges of $ G $ and two vertices of $ G^{\otimes} $ are adjacent in $ G^{\otimes} $ if the corresponding edges in $ G $ cross. Whether the crossing graph $ G^{\otimes} $ has cycles (Case 2) or not (Case 3), turns out to be an important fork in the proof. Lemma \ref{lemma:nocycles3} is a stand alone result that shows that whenever  $G^{\otimes} $ has no cycles, the graph $G$ has few edges (at most $2n-6$ to be precise) settling Case 2. Case 3 is subdivided according to the size $j$ of the smallest cycle in $G^{\otimes}$: Case 3.1 when $j=3$, Case 3.2 when $j=4$ (which includes two long subcases that only apply when $\ell=4$ and are differed to the Appendix), and Case 3.3 when $j\geq 5$ (the cases $\ell=2$ and $3$ are concluded soon after this case starts, the rest of the proof is for $\ell=4$). In all these cases, our strategy is to find what we call a \emph{valid replacement} (formally defined later), which is a set of edges in $G$ that can be replaced by one of the graphs in the previous section ---obtaining a new graph $G'$--- without increasing the local crossing number or decreasing the number of edges and in such a way that $G'$ has an uncrossed edge. This new graph falls into Case 1, which allows us to conclude the proof.  

\begin{figure}
    \centering
    \includegraphics{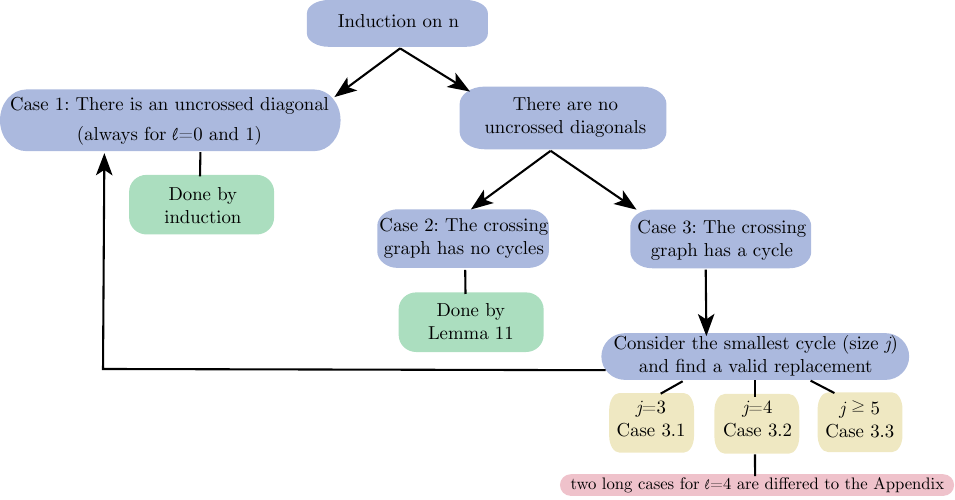}
    \caption{Organization of the proof of Theorem \ref{theorem:epsilons}.}
    \label{fig:organization}
\end{figure}
  
	\paragraph{Case 1.} Suppose $ G $ has an uncrossed diagonal, say $ v_1v_t $ with $ 3\leq t \leq n-1 $. Consider the two subgraphs $ G_1 $ and $ G_2 $ of $ G $ induced by the sets of vertices $ \{v_1,v_2,\ldots,v_t\} $ and $ \{v_t,v_{t+1},\ldots,v_{n-1},v_n,v_1\} $, respectively, removing the edge $ v_1v_t $ if present. Then $ G_1 $ and $ G_2 $ are subgraphs of $ D_t $ and $ D_{n+2-t} $, respectively,  with at most $ \ell $ crossings per edge. Therefore, each edge of $ G $ is also an edge of $ G_1 $ or $ G_2 $ except possibly for $ v_1v_t $, and by induction
	
	\begin{multline*}
	e(G) \leq 
	 e(G_1)+e(G_2)+1 \leq e_\ell(t) + e_\ell(n-t+2) + 1\\
	\leq
	\left\{
     \begin{array}{ll}
        C_\ell(t-3)+\delta_\ell(t)+C_\ell(n-t+2-3)+\delta_\ell(n-t+2)+1 &
        \begin{array}{l}
        \textup{if } 0\leq \ell \leq 3 \textup{ or }\\ \ell=4, 5\leq t \leq n-3,\vspace{.1in}
        \end{array}
        \\
        e_4(n-2) + e_4(4) + 1=\tfrac{5}{2}(n-3)+\delta_4(n-2) -2 &
        \textup{if } \ell=4, t=4 \textup{ or }t= n-2\vspace{.1in}\\
         e_4(n-1) + e_4(3) + 1=\tfrac{5}{2}(n-3)+\delta_4(n-1)  -\tfrac{3}{2} & \textup{if } \ell=4, t=3 \textup{ or }t=n-1,\vspace{.1in}\\
     \end{array}
   \right.
     \\
	= C_\ell(n-3)+\delta_\ell(n)+
		\left\{
     \begin{array}{ll}
        \delta_\ell(t)+\delta_\ell(n-t+2)-\delta_\ell(n)-C_\ell+1 &
        \begin{array}{l}
        \textup{if } 0\leq \ell \leq 3 \textup{ or }\\ \ell=4, 5\leq t \leq n-3,\vspace{.1in}
        \end{array}
        \\
         \delta_4(n-i)-\delta_4(n) -\tfrac{i+1}{2}& 
         \begin{array}{l}
         \textup{if } \ell=4, i=1 \textup{ or }2,\\ \{t,n+2-t\}=\{2+i,n-i\}
        \end{array}
     \end{array}
      \right.
	\\
	\end{multline*}
	Note that $n\geq 8$ when $\ell=4$ and so, for $i=1$ or $2$, $\delta_4(n-i)-\delta_4(n)\leq 3/2-0\leq (i+1)/2$. For the remaining cases, the inequality $ \delta_\ell(t)+\delta_\ell(n-t+2)\leq \delta_\ell(n)+C_\ell-1 $ can be easily verified.
	
Case 1 always holds for $ \ell=0 $ and $ n\geq 4 $. We claim that Case 1 also always holds for $ \ell=1$ and $ n\geq 5$. Indeed, we can assume without loss of generality that for some indices $ 1<i_1<i_2<i_3<i_4\leq n $ the edges $ v_{i_1}v_{i_3} $ and $ v_{i_2}v_{i_4} $ are in $ G $, that is, vertex $ v_1 $ does not participate in this crossing. In this case, the diagonal $ v_{i_4}v_{i_1} $ cannot be crossed by any edge in $ G $. This is because such an edge would also intersect $ v_{i_1}v_{i_3} $ or $ v_{i_2}v_{i_4} $, but $ v_{i_1}v_{i_3} $ and $ v_{i_2}v_{i_4} $ already cross each other and so they cannot cross any other edge. 


\underline{This concludes the proof for $ \ell=0$ and $\ell=1 $.} We now assume that $ 2\leq \ell \leq 4 $ and that $ G $ has no crossing-free diagonals. 

\paragraph{Case 2.} Suppose $ G^{\otimes} $ has no cycles and $ \ell\geq 2 $. Lemma \ref{lemma:nocycles3} below implies that $ e(G) \leq 2n-6 \leq C_\ell(n-3)+\delta_\ell(n) $.

\begin{lem}\label{lemma:nocycles3}
Let $ G $ be a subgraph of $ D_n $, $ n\geq 3 $. 
If $ G^{\otimes} $ has no cycles, then $ e(G)\leq 2n-6 $.
\end{lem}

\begin{proof}
Let $ e^*(n) $ be the maximum number of edges in a subgraph $ G $ of $ D_n $ such that there are no cycles in $ G^{\otimes} $. We actually prove the identity $ e^*(n) =2n-6$. The graph whose edges are $ v_1v_i $ and $ v_{i-1}v_{i+1}$ for $ 3\leq i \leq n-1 $ (Figure \ref{fig:best_e3_nocycles}a) has $ 2n-6  $ edges and satisfies the conditions of the lemma, showing that $ e^*(n) \geq 2n-6$. We prove that $ e^*(n) \leq 2n-6$ by induction on $ n $. The result is clearly true if $ n=3 $ or $ 4 $. Let $ G $ be a graph with the required properties with $ n\geq 5 $ vertices and with the maximum number of edges. Because there are no cycles in $ G^{\otimes} $, it follows that there must be an edge $ uv $ in $ G $ crossed at most once. Consider the subgraphs $ G_1 $ and $ G_2 $ of $ G $ on each side of $ uv $, each including the vertices $ u $ and $ v $ but not the edge $ uv $. Let $ n_1 $ and $ n_2 $ be the number of vertices of $ G_1 $ and $ G_2 $, respectively. Note that the graphs $ G_1 $ and $ G_2 $ inherit the conditions of the lemma from $ G $. By induction, and since $ n_1+n_2=n+2 $ and the only edges that are not part of either $ G_1 $ or $ G_2 $ are $ uv $ and the edge (if any) crossing $ uv $, 
	$$
	e(G)\leq e(G_1)+e(G_2)+2\leq 2n_1-6 +2n_2-6 +2=2n-6.
	$$
\end{proof}

\begin{figure}[h]
	\centering
	\includegraphics[width=0.8\linewidth]{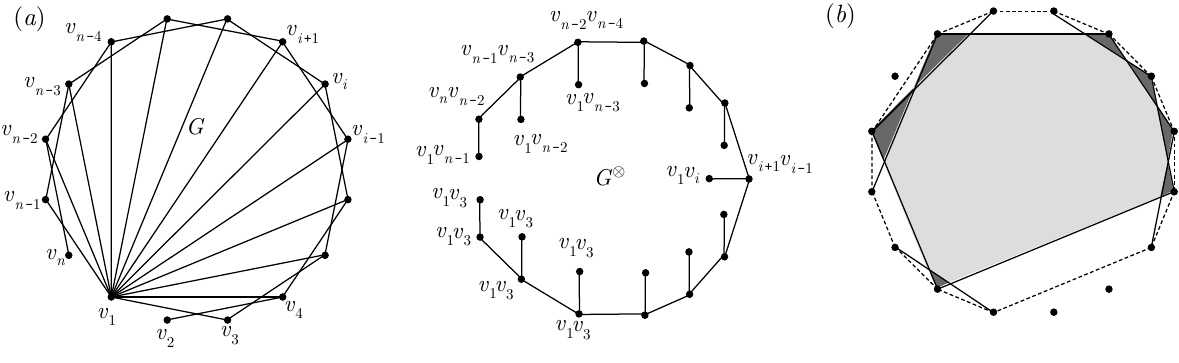}
	\caption{(a) Optimal graph for $ e^*(n). $ (b) The corners (dark shade) and the polygon (light shade) whose vertices are the crossings in the cycle $ C $.}
	\label{fig:best_e3_nocycles}
\end{figure}

\paragraph{Case 3.} Suppose that $ G^{\otimes} $ has a cycle and $ \ell\geq 2 $. In this case, we modify $ G $ to obtain a graph $ G' $ on the same vertex-set, 
with local crossing number at most $ \ell $, and such that $e(G)\leq e(G')\leq C_\ell(n-3)+\delta_\ell(n)$.

 Given a set of vertices $ U $, we denote by $ \conv(U) $ the convex hull of $ U $ and by $ \bd(U) $ the boundary of $ \conv(U) $. We say that $ G $ has a \emph{valid replacement} if there is a subset $ U $ of at least 3 vertices of $ G $  such that 
 \begin{equation}\label{eq:generatedplusborder}
  |\gen(U)|\leq C_\ell(|U|-3)+\delta_\ell(|U|)+|\nonedge(U)|,
 \end{equation}
 where $ \gen(U) $ is the set of edges of $ G $ intersecting the interior of $ \conv(U) $ and $ \nonedge(U) $ is the set of nonedges of $ G $ on $ \partial\conv(U) $. In this case, we say that $U$ \emph{generates} a valid replacement. 
 
 If $ G $ has a valid replacement generated by $ U $, then we obtain $ G' $ by
 \begin{itemize}
     \item removing $ \gen(U) $, leaving the interior of $ \conv(U)$ empty;
     \item adding the elements of $ \nonedge(U) $ as edges of $G'$; and
     \item adding a copy of a graph with vertex-set $ U $, local crossing number at most $ \ell $, and $ C_\ell(|U|-3)+\delta_\ell(|U|) $ edges (which we have proved exists).
 \end{itemize} 
 Note that $ G' $ has at least as many edges as $ G $ and has local crossing number at most $ \ell $  because the added edges are contained in $ \conv(U) $ and so they do not cross any edges of $G$ not in $\gen(U) $. 
 
If $ |U|=n $, then $e(G)=|\gen(U)|$  and so  the entire graph $G$ is being replaced by a graph $G'$ so that $e(G)\leq e(G')= C_\ell(n-3)+\delta_\ell(n)$. (In fact, since $|\nonedge(U)|=0$ when $|U|=n$, Inequality (\ref{eq:generatedplusborder}) already shows that $e(G)\leq C_\ell(n-3)+\delta_\ell(n)$.)

If $ |U|<n $, then there is at least one diagonal on $ \partial\conv(U) $ that is not a side of $G$. This diagonal becomes an uncrossed diagonal of $G'$ and so $ e(G)\leq e(G') \leq C_\ell(n-3)+\delta_\ell(n) $  by Case 1. 
For future reference, $ |U|<n $ whenever there is an edge in $ \gen(U) $ that crosses $ \partial\conv(U) $.

Let $ C $ be the smallest cycle in $ G^{\otimes} $ and $ j $ its size. $ C $ corresponds to a sequence of edges $ a_1,a_2,\ldots,a_j $ in $ G $ such that edge $ a_i $, with endpoints $ p_i $ and $ q_i $, crosses the edges $ a_{i-1} $ and $ a_{i+1} $ (the subindices are taken mod $ j $). The edges $ a_{i-1} $ and $ a_{i+1} $ can be incident to the same vertex or disjoint. In the first case, we say that the triangle formed by $ a_{i-1} $, $ a_i $, and $ a_{i+1} $ is a \emph{corner} of $ C $ (see Figure \ref{fig:best_e3_nocycles}b). Note that, by minimality of $ C $, the edges $ a_{i-1} $ and $ a_{i+1} $ can cross only if $ j=3 $, in which case $ C $ has no corners. Let $ V= \{p_1,p_2,\ldots,p_j,q_1,q_2,\ldots,q_j \}$, $ c $ be the number of corners of $ C $, $ E=\{a_1,a_2,\ldots,a_j\} $, $ D $ be the set of edges of $ G $ not in $ E $ that cross at least one edge in $ E $, $E'=E\cup D$, $ n' =|V|$, and $ e'=|E'| $. Note that the edges of $ G $ that are not in $ E' $ cannot cross $\conv(V)$. 

If $ e'\leq C_\ell(n'-3)+\delta_\ell(n') $, then $ V $ generates a valid replacement with $ |\gen(V)|=|E\cup D|=e' $. We now analyze under which circumstances we can guarantee that 
\begin{equation}\label{eq:need}
e'\leq C_\ell(n'-3)+\delta_\ell(n')
\end{equation} 
and therefore the existence of  a valid replacement. Note that $ n'=2j-c $ and $ |D|\leq j(\ell-2) $ because each edge in $ E $ already crosses two other edges in $ E $ and so it can only cross at most $ \ell -2 $ edges in $ D $. Then $ e'\leq j(\ell-1) $.

\paragraph{Case 3.1} Suppose $ j=3 $. Then $ c=0 $ and thus $ n'=6 $ and $ e'\leq 3(\ell-1)\leq 3C_\ell+\delta_\ell (6) = C_\ell(n'-3)+\delta_\ell (n')$.

\paragraph{Case 3.2} Suppose $ j=4 $. Figure \ref{fig:Casej4N} shows all possible cycles $ C $ formed by 4 edges in $ G $. They satisfy that $ 6\leq n'\leq 8 $ and $ e'\leq 4(\ell-1) $. It can be verified that Inequality (\ref{eq:need}) holds for $ 2\leq \ell\leq 4 $ and $ 6 \leq n' \leq 8 $ except when $ \ell=4 $, $ n'=7 $, and $ e'=12 $ (Figure \ref{fig:Casej4N}b); or $ \ell=4 $, $ n'=6 $, and $ 10\leq e' \leq 12 $ (Figure \ref{fig:Casej4N}c). (In all other cases $ C_\ell(n'-3)+\delta_\ell (n')\geq 4(\ell-1) $.) These remaining cases (all for $ \ell =4 $) involve a long and careful case analysis. The details are included in Appendix \ref{sect:case3.3}.

\begin{figure}[h]
	\centering
	\includegraphics[width=.7\linewidth]{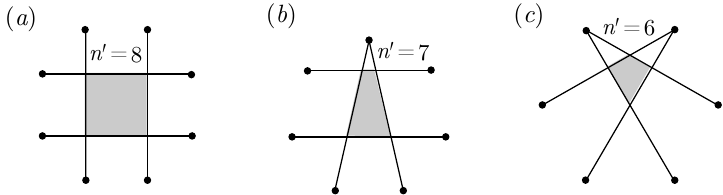}
	\caption{Possible cycles $ C $ with  $ j=4 $.}
	\label{fig:Casej4N}
\end{figure}

\paragraph{Case 3.3} Suppose $ j\geq 5 $.  The subgraph of $ G $ whose edge set is $ E $ has vertices of degree 2, the corners, and vertices of degree 1, the leaves. The edges of $ E $ are of three types: \emph{type 2} edges join two leaves, \emph{type 1} edges join a leaf and a corner, and \emph{type 0} edges join two corners. Let $ t_2, t_1, $ and $ t_0 $ be the number of edges in $ E $ of type 2, 1, and 0, respectively. By minimality of $ C $, any edge in $ E $ crosses at most one edge in $ D $ and any edge in $ E $ joining two corners can only be crossed by edges in $ E $. That is, edges of type 0 are not crossed by edges in $ D $; and each edge of type 2 or 1 is crossed by at most $ \ell-2 $ edges in $ D $. Then $ |D|\leq (\ell-2)(t_2+t_1)$. Note that $ t_1\leq 2c $ because each corner is incident to at most two edges of type 1. Also $2t_2+t_1=n'-c$ (all vertices of degree 1 in the graph induced by $ C $ counted in two different ways) and $ n'=2j-c $ (the sum of the degrees in the graph induced by $ C $ is $ 2\cdot c +1\cdot (n'-c)=2j$). So  $2t_2+t_1=2(j-c) $ and
\begin{eqnarray}
e'&=&|D\cup E|=|D|+j\leq (\ell-2)(2t_2+t_1)-(\ell-2)t_2+j\nonumber\\
&=&2(\ell-2)(j-c)-(\ell-2)t_2+j=(2\ell-3)j-2(\ell-2)c-(\ell-2)t_2.
\end{eqnarray}

Thus $ (2\ell-3)j-2(\ell-2)c-(\ell-2)t_2\leq C_\ell(n'-3)+\delta_\ell(n')=C_\ell(2j-c-3)+\delta_\ell(2j-c) $ guarantees (\ref{eq:need}). This is equivalent to 
\begin{equation}\label{eq:need2all}
(2C_\ell-2\ell+3)j+(2\ell-4-C_\ell)c-3C_\ell+(\ell-2)t_2+\delta_\ell(2j-c) \geq 0.
\end{equation}

If $ \ell=2 $, Inequality (\ref{eq:need2all}) becomes 
\begin{equation}\label{eq:need_ell2}
3j-2c-6+\delta_2(2j-c)\geq 0.
\end{equation}
Since $ j\geq 5 $, $ j\geq c $, and $ \delta_2(2j-c)\geq 0 $, then $3j-2c-6+\delta_2(2j-c)\geq (j-5)-2(j-c)+\delta_2(10-c)-1\geq 0$ because $j-5=j-c= \delta_2(10-c)=0$ never happens ($j=5=c$ implies $\delta_2(10-c)=\delta_2(5)=1$).

If $ \ell=3 $, Inequality (\ref{eq:need2all}) becomes 
\begin{equation}\label{eq:need_ell3}
\tfrac{1}{4}(6j-c-27)+t_2+\delta_3(2j-c)\geq 0.
\end{equation}
Since $ j\geq c $, $ t_2\geq 0 $, and $ \delta_3(2j-c)\geq -\frac{1}{4} $, then $\frac{1}{4}(6j-c-27)+t_2+\delta_3(2j-c)\geq \frac{1}{4}(6j-j-27-1)=\frac{1}{4}(5j-28)\geq 0$ for $ j\geq 6 $. If $ j=5 $, Inequality (\ref{eq:need_ell3}) becomes $ \frac{1}{4}(3-c)+t_2+\delta_3(10-c)\geq 0$, which holds because $ \frac{1}{4}(3-c)+\delta_3(10-c)\geq 0 $ for any $ 0\leq c \leq 5. $

\underline{This concludes the proof for $ \ell=2 $ and $ \ell=3 $.} 

If $ \ell=4 $, Inequality (\ref{eq:need2all}) becomes 
\begin{equation}\label{eq:need2}
\tfrac{3}{2}(c-5)+2t_2+\delta_4(2j-c)\geq 0.
\end{equation}
Inequality (\ref{eq:need2}) holds for  $c\geq 5 $ as $ t_2 $ and $ \delta_4 (2j-c) $ are nonnegative. Note that  $$ 5\leq j=\tfrac{1}{2}(2c+t_1+2t_2)\leq \tfrac{1}{2}(2c+2c+2t_2)=2c+t_2. $$ Then $\tfrac{3}{2}(c-5)+2t_2+\delta_4(2j-c)\geq \tfrac{3}{2}(c-5)+2(5-2c)+\delta_4(2j-c)=\delta_4(2j-c)-\tfrac{5}{2}(c-1)\geq 0$ for $c=0$ or 1. When $2\leq c \leq 4$, there are possible values of $t_2$ and $j\geq 5$ for which Inequality (\ref{eq:need2}) might not hold:
\begin{itemize}
\item If $ c=4 $, Inequality (\ref{eq:need2}) becomes $ 2t_2-\frac{3}{2}+\delta_4(2j-4) \geq 0 $ with $ \delta_4(2j-4)=\frac{3}{2} $ or $\frac{1}{2} $ if $ j $ is odd or even, respectively. The only case in which (\ref{eq:need2}) does not hold is when $ j $ is even and $ t_2=0 $. Since $5\leq j \leq 2c+t_2=8$, then $ j=6 $ or $ j=8 $. Figure \ref{fig:CaseC4} shows the only possibilities for $ C $.
\begin{figure}[h]
\centering
\includegraphics[width=0.5\linewidth]{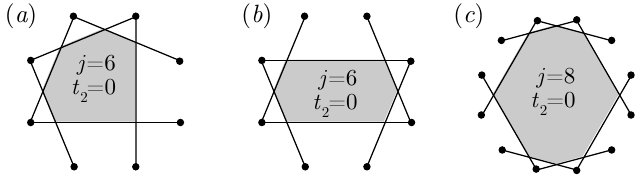}
\caption{Remaining possibilities in Case 3.3 with $ \ell =4 $ and  $ c=4. $}
\label{fig:CaseC4}
\end{figure}

\item If $ c=3 $, Inequality (\ref{eq:need2}) is equivalent to $ 2t_2-3+\delta_4(2j-3) \geq 0 $ with $ \delta_4(2j-3)=1 $ if $ j $ is odd and $ \delta_4(2j-3)=0 $ if $ j $ is even. So the only cases in which (\ref{eq:need2}) does not hold are when $ j $ is odd and $ t_2=0 $, or when $ j $ is even and $ t_2=0 $ or 1. 
Since $5\leq j \leq 6+t_2$, then $ j=5 $ and $t_2=0$; or $ j=6 $ and $t_2=0$ or $1$. The only possibilities for $ C $ are shown in Figure \ref{fig:CaseC3}.
\begin{figure}[h]
\centering
\includegraphics[width=0.5\linewidth]{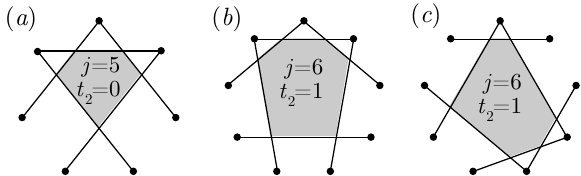}
\caption{Remaining possibilities in Case 3.3 with $\ell=4$ and $ c=3. $}
\label{fig:CaseC3}
\end{figure}
\item If $ c=2 $, Inequality (\ref{eq:need2}) becomes $ 2t_2-\frac{9}{2}+\delta_4(2j-2) \geq 0 $ and $ \delta_4(2j-2)=\frac{1}{2}$ or  $\frac{3}{2}$. The only cases in which (\ref{eq:need2}) does not hold is when $ t_2=0 $ or 1. Since 
$ 5\leq j\leq 4+t_2 $, then $ t_2=1 $ and $ j=5 $. Figure \ref{fig:CaseC2} shows the only possibility for $ C $.
\begin{figure}[h]
\centering
\includegraphics[width=0.15\linewidth]{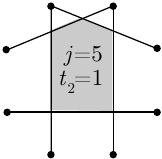}
\caption{Remaining possibilities in Case 3.3 with $\ell=4$ and $ c=2. $}
\label{fig:CaseC2}
\end{figure}
\end{itemize} 
Table \ref{table:cases} lists 
the cases where $ j\geq 5 $ and Inequality (\ref{eq:need}) is not necessarily satisfied, including the values of $ n'$ and $e' $. Note that by minimality of $ C $, each edge in $ D $ crosses exactly one edge in $ E $ and can be crossed by up to three other edges of $ G $. The tick marks in each figure indicate the possible crossings of edges in $ D $ with edges in $ E $. Note that all such crossings, except at most one of them in Figures  \ref{fig:CaseC3}a and \ref{fig:CaseC2}, must happen in order for Inequality (\ref{eq:need}) to fail. All crossings with type 1 edges must happen exactly at the indicated place. However, the two crossings with type 2 edges could happen on the same side of the edge (although Table \ref{table:cases} shows one tick mark per side).

\begin{table}[h]
\begin{center}
\scalebox{0.7}{
\begin{tabular}{|c|c|c|c|c|c|c|c|}
   \hline 
& 
\includegraphics[width=0.14\linewidth]{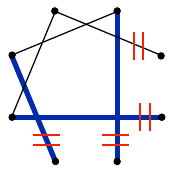} &
\includegraphics[width=0.14\linewidth]{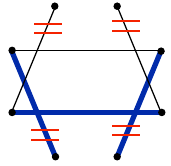} &
\includegraphics[width=0.14\linewidth]{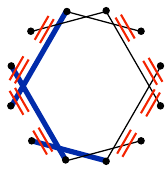} &
\includegraphics[width=0.14\linewidth]{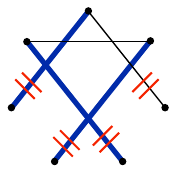} &
\includegraphics[width=0.14\linewidth]{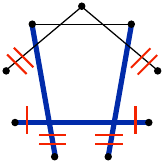} &
\includegraphics[width=0.14\linewidth]{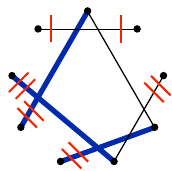} &
\includegraphics[width=0.14\linewidth]{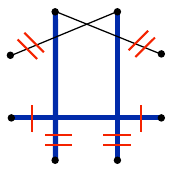}
  \\
   \hline 
Case & Figure \ref{fig:CaseC4}a & Figure \ref{fig:CaseC4}b & Figure \ref{fig:CaseC4}c & Figure \ref{fig:CaseC3}a & Figure \ref{fig:CaseC3}b & Figure \ref{fig:CaseC3}c  & Figure \ref{fig:CaseC2}  \\
   \hline 
$ n' $ & 8 & 8 & 12 & 7 & 9 & 9  & 8  \\
   \hline 
$ \frac{5}{2}(n'-3)+\delta_4(n') $ & 13 & 13 & 23 & 11 & 15 & 15  & 13  \\
   \hline 
$ e'  $ & $ 6+8=14 $ & $ 6+8=14 $ & $ 8+16=24 $ &  \begin{tabular}{@{}c@{}c@{}}$ 5+7=12 $\\or \\ $ 5+8=13 $ \end{tabular}& $ 6+10=16 $ & $ 6+10=16 $  & \begin{tabular}{@{}c@{}c@{}}$ 5+9=14 $ \\or\\ $ 5+10=15 $ \end{tabular}  \\
\hline
$ n_1 =n'+2$ & 10 & 10 & 14 & 9 & 11 & 11  & 10  \\
   \hline 
$ \frac{5}{2}(n_1-3)+\delta_4(n_1) $ & 19 & 19 & 29 & 15 & 21 & 21  & 19  \\
   \hline
$ e_1 \leq e'+3 $ & 17 & 17 & 27 & 15 or 16& 19 & 19  & 17 or 18 \\
\hline
\end{tabular}
}
\end{center}
\caption{ Remaining cases for $\ell=4$ and $ j\geq 5 $.}
\label{table:cases}
\end{table}

Consider an edge $ xy\in D $. If $ x\notin V $ and $ y\notin V $, let $ V_1=V\cup \{x,y\} $ and $ E_1=\gen(V_1) $, that is, $ E_1 $ is the union of $ E' $ and any edges crossing $ xy $. Let $ n_1=|V_1| =n'+2$ and $ e_1=|E_1|\leq e'+3 $. Whenever $ e_1\leq \frac{5}{2}(n_1-3)+\delta_4(n_1) $,  $ V_1 $ generates a valid replacement. Note that in all cases shown in Table \ref{table:cases}, except for Figure \ref{fig:CaseC3}a with $ e'=13 $, $ e_1\leq e'+3\leq \frac{5}{2}(n_1-3)+\delta_4(n_1) $.  We treat Figure \ref{fig:CaseC3}a with $ e'=13 $ separately at the end of the section. For all other cases (including Figure \ref{fig:CaseC3}a with $ e'=12 $), we argue that such an edge $xy$ always exists, that is, it is not possible for each edge in $ D $ to be incident to at least one vertex in $ V $. In each of these cases, there are three consecutive edges $ a_{i-1} $, $ a_i $, and $ a_{i+1} $ in $ C $ (a possible choice is highlighted in Table \ref{table:cases}) that satisfy the following conditions, refer to Figure \ref{fig:LastPart1}a: (1) $ a_{i-1} $ and $ a_{i+1} $ do not have vertices in common (in fact we only need $ q_{i-1} \neq p_{i+1} $), (2) $ a_{i-1} $ and $ a_{i+1} $ are of type 1, (3) $ a_{i-1} $ is crossed by at least one edge in $ D $, call it $ b $, and (4)  $ a_{i+1} $ is crossed by two edges in $ D $, call them $ c $ and $ c' $.
\begin{figure}[h]
\centering
\includegraphics[width=.9\linewidth]{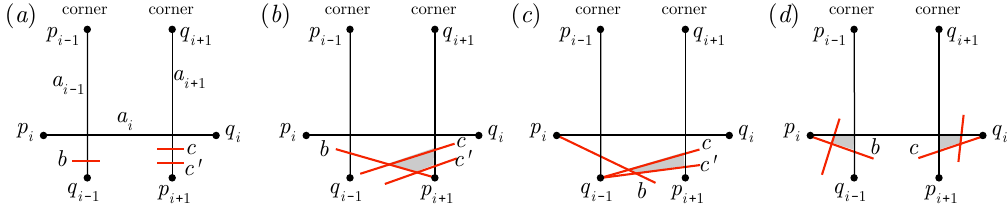}
\caption{Three consecutive edges in $ E $. Tick marks represent edges in $ D $.}
\label{fig:LastPart1}
\end{figure}

By minimality of $ C $, there are no triangles or quadrilaterals in $G^{\otimes} $. If each edge in $ D $ were incident to at least one vertex in $ V $, then  $ b $ would be incident to $ p_i $ or $ p_{i+1} $. But $ b $ is not incident to $ p_{i+1} $ (Figure \ref{fig:LastPart1}b), because $ b $, $ c, c' $, and $ a_{i+1} $ would form a quadrilateral (or a triangle if $c$ and $c'$ cross) in $ G^{\otimes} $. So $ b $ is incident to $ p_i $. If both $ c $ and $ c' $ are incident to $ q_{i-1} $ (Figure \ref{fig:LastPart1}c), then $ c,c',b $, and $ a_{i+1} $ would form a quadrilateral. But if $ c $ is incident to $ q_i $ (Figure \ref{fig:LastPart1}d), one of the edges crossing $ a_i $ would form a quadrilateral with either $ a_i,a_{i-1} $, and $ b $ or with $ a_i,a_{i+1} $, and $ c $. (The argument still applies if the roles of $ a_{i-1} $ and $ a_{i+1} $ are reversed.)

In Figure \ref{fig:CaseC3}a with $ e'=13 $, all edges in $ D $ shown by tick marks in Table \ref{table:cases} are present. If two edges $ wx $ and $ yz $ in $ D $ are not incident to vertices in $ V $, then let $ V_2=V\cup \{w,x,y,z\}$ and $ E_2 $ be the union of $ E' $ and any edges crossing $ wx $ or $ yz $. Then $ |V_2|=10 $ or $11 $ and $ |E_2|\leq e' +6=19$, and thus $ V_2 $ generates a valid replacement. Finally, the two edges in $ D $ circled in Figure \ref{fig:specialcasej5}a cannot be incident to vertices in $ V $, because a quadrilateral would be formed in $ G^{\otimes} $ (Figures \ref{fig:specialcasej5}b-c).

\begin{figure}[h]
\centering
\includegraphics[width=.5\linewidth]{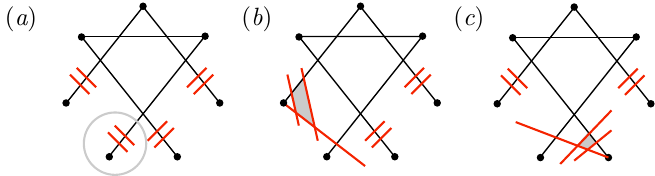}
\caption{Two edges in Figure \ref{fig:CaseC3}a that are not incident to vertices in $ V $.}
\label{fig:specialcasej5}
\end{figure}

\appendix
\section{Proof of remaining cases for $ j=4 $ and $ \ell=4 $ in Theorem \ref{theorem:epsilons}}\label{sect:case3.3}

Figure \ref{fig:Casej4} shows all remaining cases for $ j=4 $ and $ \ell=4 $: Figure \ref{fig:Casej4}a with $ e'=12 $ (all tick marks represent edges in $ D $) and Figure \ref{fig:Casej4}b with $ 10\leq e'\leq 12 $ (one or two tick marks could be missing).

\begin{figure}[h]
\centering
\includegraphics[width=.45\linewidth]{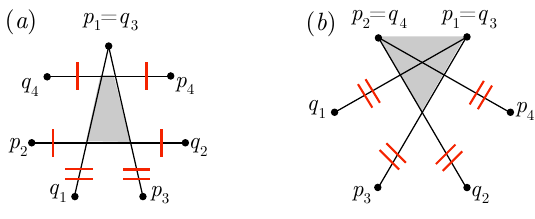}
\caption{Remaining cases for $ j=4 $ and $ \ell=4 $. (a) $ e'=12 $. (b) $ 10\leq e' \leq 12 $.}
\label{fig:Casej4}
\end{figure}

We use the following observation several times along the proof.

\begin{lem}\label{obs:boundary}
The set of vertices $ V $ in Figure \ref{fig:Casej4}a generates a valid replacement whenever there is at least one nonedge in $ \bd (V) $. The set of vertices $ V $ in Figure \ref{fig:Casej4}b generates a valid replacement whenever there are at least $ e'-9 $ nonedges in $ \bd (V) $.
\end{lem}

\begin{proof}
 In Figure \ref{fig:Casej4}a, $ |V|=n'=7 $, $ |\gen(V)|=e'=12 $, and assume that $ |\nonedge(V)|\geq 1$. Then $$ e'=12=\frac{5}{2}(7-3)+1+1\leq C_4(n'-3)+\delta_4(n')+|\nonedge(V)|. $$ 
 In Figure \ref{fig:Casej4}b, $ |V|=n'=6 $, $ |\gen(V)|=e' $, and assume that $ |\nonedge(V)|\geq e'-9$. Then $$ e'=\frac{5}{2}(6-3)+\frac{3}{2}+(e'-9)\leq C_4(n'-3)+\delta_4(n')+|\nonedge(V)|. $$ 
 Both cases satisfy (\ref{eq:generatedplusborder}), thus guaranteeing a valid replacement generated by $ V $.
 \end{proof}


\paragraph{Consider Figure \ref{fig:Casej4}a.} Because $ e'=12 $, each edge in $ D $ crosses exactly one edge in $ E $ and so it does not cross the shaded region. If there is an edge $ xy\in D $ with $ x\notin V $ and $ y\notin V $, let $ V_1=V\cup \{x,y\} $ and $ E_1=\gen(V_1) $, that is, $ E_1 $ is the union of $ E' $ and any edges crossing $ xy $. Then $ |V_1| =9$ and $ |E_1|\leq e'+3=15$. So $ V_1 $ generates a valid replacement. Similarly, if there is an edge $ xy\in D $ with $ x\notin V $ and $ y\in V $ that crosses at least two other edges in $ D $, let $ V_1=V\cup \{x\} $ and $ E_1 =\gen(V_1)$. Then $ |V_1| =8$ and $ |E_1|\leq e'+1=13$ (because there is at most one  edge crossing $ xy $ that is not in $ E\cup D $). So $ V_1 $ generates a valid replacement. So assume that any edge in $ D $ is incident to at least one vertex in $ V $ and that any edge incident to exactly one vertex in $ V $ crosses at most one edge in $ D $. 

\begin{figure}[h]
	\centering
	\includegraphics[width=.6\linewidth]{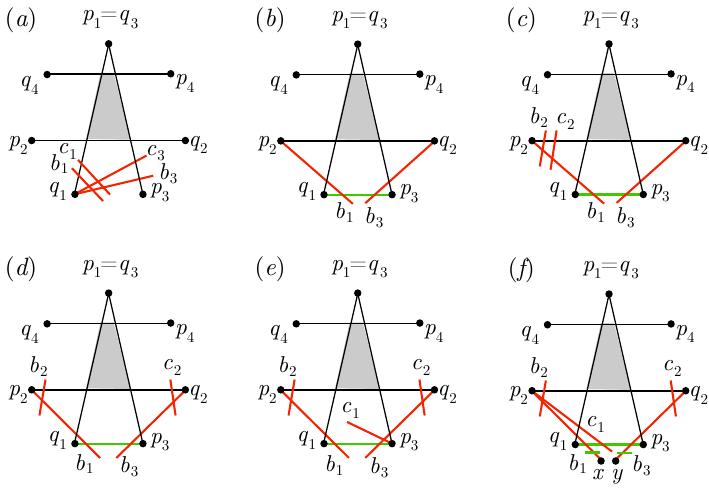}
	\caption{Completing the case corresponding to Figure \ref{fig:Casej4}a.}
	\label{fig:Casej4_a}
\end{figure}

Let $ b_i $ and $ c_i $ be the two edges in $ D $ crossing $a_i=p_iq_i $, $ 1\leq i \leq 4 $. Then each of $ b_1 $ and $ c_1 $ must be incident to $ p_2 $ or $ p_3 $ (or both); and each of $ b_3 $ and $ c_3 $ must be incident to $ q_1 $ or $ q_2 $ (or both). Suppose first that $ b_3 $ and $ c_3 $ are both incident to $ q_1 $. (Figure  \ref{fig:Casej4_a}a.) Then  $ b_1 $ and $ c_1 $ cross $ b_3 $ and $ c_3 $ and at least one of $ b_1 $ or $ c_1 $ is incident to exactly one vertex of $ V $ contradicting the assumptions. Then we can assume that $ b_3 $ is incident to $ q_2 $ and not to $ q_1 $; and (by symmetry) $ b_1 $ is incident to $ p_2 $ and not to $ p_3 $. (Figure \ref{fig:Casej4_a}b.) Note that the diagonal $ q_1p_3 $ is not a side (it is crossed by $ b_1 $ and $ b_3 $). If $ q_1p_3 $ is a nonedge, then $ V $ generates a valid replacement by Lemma \ref{obs:boundary}. Assume that $ q_1p_3 $ is an edge. Note that $ b_2 $ must cross $ b_1 $ or $ b_3 $ depending on where $ b_2 $ crosses $ a_2 $. The same applies to $ c_2 $. But both $ b_2 $ and $ c_2 $ cannot simultaneously cross $ b_1 $ (or $ b_3 $) as this would contradict our hypotheses ($ b_1 $ is incident to exactly one vertex in $ V $ and it would be crossed by two other edges in $ D $, Figure \ref{fig:Casej4_a}c). Then assume that $ b_2 $ crosses $ b_1 $ and $ c_2 $ crosses $ b_3 $.  (Figure \ref{fig:Casej4_a}d.) Similarly, $ c_1 $ cannot be incident to $ p_3 $ (Figure \ref{fig:Casej4_a}e), as then $ b_3 $ (which is incident to exactly one vertex in $ V $) would be crossed by $ c_1 $ and $ c_2 $. Thus $ c_1 $ is incident to $ p_2 $ and not to $ p_3 $. Since $ b_1 $ and $ c_1 $ satisfy the same conditions so far, we can assume without loss of generality that $ b_1 $ and $ b_3 $ do not share any vertices (Figure \ref{fig:Casej4_a}f). Say that $ b_1=p_2x $ and $ b_3=q_2y $ for some vertices $ x $ and $ y $ not in $ V $. Let $ V_1=V\cup \{x,y\} $ and $ E_1=\gen(V_1) $. Then $ |V_1| =9$ and $ |E_1|\leq e'+3=15$ because $ q_1p_3 $ crosses both $ b_1 $ and $ b_3 $, and so there is at most one more edge not in $ E\cup D$ crossing $ b_1 $ and one more crossing $ b_3 $. So $ V_1 $ generates a valid replacement.


\paragraph{Consider Figure \ref{fig:Casej4}b.} As before, let $ b_i $ and $ c_i $ be the edges in $ D $ (if they exist) crossing $a_i=p_iq_i $, $ 1\leq i \leq 4 $. If only one edge in $ D $ crosses $a_i$, we call it $b_i$. We assume all of the following conditions:
\begin{itemize}
	\item $ 10\leq e' \leq 12 $,
	
	\item there is no set of four edges in $ G $ crossing like in Figures \ref{fig:Casej4N}a-b, (if there were such a set, we would use that cycle instead of $ C $ to guarantee a valid replacement),
	
	\item there are at most $ e'-10 $ nonedges in $ \bd (V) $ (if there were at least $ e'-9 $ such nonedges, Lemma \ref{obs:boundary} would guarantees a valid replacement),
	
	\item no edge in $ D $ crosses the shaded region.
	
	Indeed, if such edges exist and they are incident to $ p_1 $ or $ p_2 $, then we replace $ C $ by the two edges in $ E' $ incident to $ p_1 $ closest to $ p_1p_2 $ and  the two edges in $ E' $ incident to $ p_2 $ closest to $ p_1p_2 $. Any edge crossing the shaded region after this replacement is not incident to $ p_1 $ or $ p_2 $ and so it would either create a triangle in $G^{\otimes} $ or 4 edges crossing as in Figure \ref{fig:Casej4N}b.

\end{itemize}

We consider three cases according to how the edges in $ D $ cross the edges in $ E $.

\paragraph{Case 1.} There are at least two edges in $ D $ each crossing more than one edge in $ E $.

Note that such edges must be incident to $ p_1 $ or  $p_2$. Say that $ ux \in D$ and $ vy \in D$ with $ \{u,v\}\subseteq\{p_1,p_2\} $ are crossed by more than one edge in $ E $. (See Figure \ref{fig:Casej4_b2}.) Then $ e'=10 $ and $ p_3q_2 $ is an edge ($ p_3q_2 $ is not a side because it is crossed by  $ ux $; and it is not a nonedge because it is on $\bd(V)$ and $ e'-10=0 $). So there is at most one edge not in $ E' \cup\{p_3q_2\}$ crossing $ ux $ and at most one crossing $ vy $. Let $ V_1=V\cup \{x,y\}$ and $ E_1=\gen(V_1) $, the union of $ E' $ and any other edges crossing $ ux $ or $ vy $. Then $ |V_1|=7 $ or $ 8 $, $ |E_1|\leq 12 $ or $ 13 $, respectively. In all cases, except when $ |V_1|=7 $ and $ |E_1| =12$, $ V_1 $ generates a valid replacement.  Assume that $ |V_1|=7 $ and $ |E_1| =12$. In this case $ x=y $ (Figure \ref{fig:Casej4_b2}c) and there is one edge not in $ E' \cup\{p_3q_2\}$ simultaneously crossing $ ux $ and $ vy $. This edge is incident to $ p_3 $ or $ q_2 $ as otherwise it would form a triangle or a quadrilateral like that in Figure \ref{fig:Casej4N}b. Suppose $ p_3w $ is such edge. Since $ e'=10 $, there is an edge $ b_3 $ in $ D $ crossing $ a_3 $. But then $ b_3 $ also crosses $ p_3w $ (otherwise $ b_3,a_3, vy $, and $ a_4 $ would form a quadrilateral like that in Figure \ref{fig:Casej4N}b) and so there is at most one edge not in $ E_1 $ crossing $ p_3w $. Let $ V_2=V_1\cup \{w\} $ and $ E_2=\gen(V_2) $, the union of $ E_1 $ and any other edges crossing $ p_3w $. Then $|V_2|=8 $ and $ |E_2|\leq 12+1=13 $. So $ V_2 $ generates a valid replacement.

\begin{figure}[h]
	\centering
	\includegraphics[width=.7\linewidth]{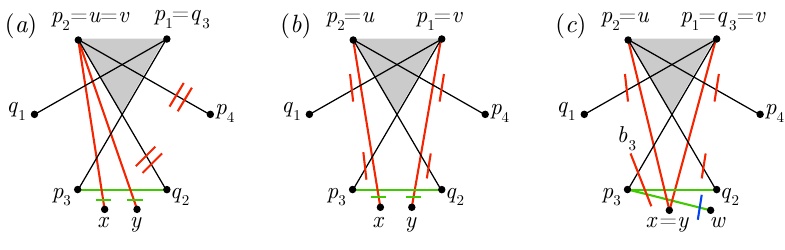}
	\caption{Two edges $ ux $ and $ vy $ in $ D $ cross more than one edge in $ E $.}
	\label{fig:Casej4_b2}
\end{figure}

\paragraph{Case 2.} There is exactly one edge in $ D $  crossing more than one edge in $ E $.

As before, such an edge must be adjacent to $ p_1 $ or $ p_2 $. Assume by symmetry that $ p_1x $ is such edge. Then $ n'=10 $ or $ 11 $. So there is at least one edge in $ D $ crossing $ a_3 $. Note that no edge in $ D $ crosses $ p_1x $ because it would create a triangle in $ G^{\otimes} $ or a quadrilateral like those in Figures \ref{fig:Casej4N}a-b. 

\begin{figure}[h]
	\centering
	\includegraphics[width=.9\linewidth]{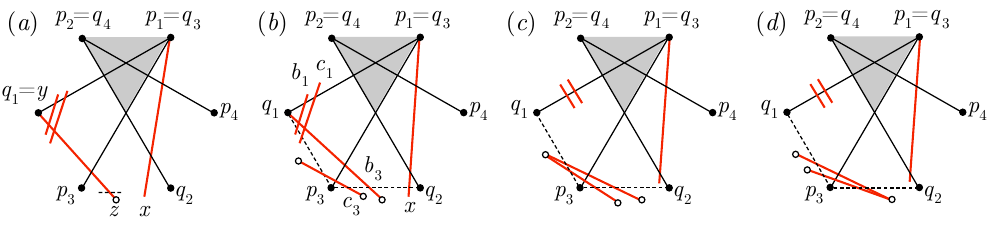}
	\caption{Possible cases for Figure \ref{fig:Casej4}b with exactly one edge in $ D $ crossing more than one edge in $ E $ and (a) exactly one edge in $ D $ crossing $ a_3 $ or (b-d) two edges in $ D $ crossing $ a_3 $ and $ n_1=9 $.}
	\label{fig:Casej4_1_9}
\end{figure}

If $b_3=yz$ is the only edge crossing $ a_3$, then $ n'=10 $ and $ a_1 $ is crossed by two edges $ b_1 $ and $ c_1 $ in $ D $. If $ y $ and $ z $ are not in $ V $, then let $ V_1=V\cup\{y,z\} $. Then $ \gen(V_1)$ is the union of $ E' $ and any other edges crossing $ yz $. Then $ |V_1|=8 $, $|\gen(V_1)|\leq 10+3=13 $, and so $ V_1 $ generates a valid replacement. Then we can assume that $ y=q_1 $ and $ z\notin V $ ($b_3$ is not incident on $q_2$ because it does not cross $ p_1x $, Figure \ref{fig:Casej4_1_9}a). Thus $ b_1 $ and $ c_1 $ also cross $ b_3 $, that is, $ b_3 $ is crossed by at most one edge not in $ E' $. Let $ V_2=V\cup\{z\} $. Then $|V_2|=7 $, $ |\gen(V_2)|\leq 10+1=11 $, and so $ V_2 $ generates a valid replacement. 

Assume now that there are two edges $ b_3 $ and $ c_3 $ in $ D $ crossing $ a_3 $. Let $ V_1 $ be the union of $ V $ and any vertices incident to $ b_3 $ or $ c_3 $. Then $ \gen(V_1) $ is the union of $ E' $ and any edges crossing $ b_3 $ or $ c_3 $. Let $ n_1=|V_1| $ and $ e_1=|\gen(V_1)| $. Then $ 8\leq n_1 \leq 10 $ ($b_3$ and $c_3$ are not incident to $q_2$ because they do not cross $p_1x$) and $ e_1\leq e'+6\leq 17 $. If $ n_1=10 $, then $ V_1 $ generates a valid replacement because $ e_1\leq 17<19=\frac{5}{2}(10-3)+\delta_4(10) $. Similarly, there is a valid replacement if $ n_1=9 $ and $ e_1\leq 15 $ or if $ n_1=8 $ and $ e_1\leq 13 $.

If $ n_1=9 $ ($b_3$ and $c_3$ span exactly 3 new vertices as in Figures \ref{fig:Casej4_1_9}b-d) and $ e_1=16 $ or $17$, then $ b_3 $ and $ c_3 $ simultaneously cross each of the diagonals $ q_1p_3 $ and $ p_3q_2 $, one of which must be an edge because  $ e'\leq 11 $. Therefore, $ e_1=16 $ and thus $e'=11 $ and there are exactly $ 5 $ edges not in $ E' $ crossing $ b_3$ or $c_3 $.  In particular, there are two edges $ b_1 $ and $ c_1 $ in $ D $ crossing $ a_1 $, and no edge in $ D $ crosses $ b_3 $ or $ c_3 $. (This takes care of Figure \ref{fig:Casej4_1_9}b.) Thus $ b_1 $ and $ c_1 $ are not incident to $ p_3 $ and so at least one of them, say $ b_1 $, is incident to a vertex not in $ V_1 $. Let $ V_2$ be the union of $ V_1 $ and the vertices spanned by $ b_1 $. Then $ |V_2|\geq 9+1=10 $, $ |\gen(V_2)|\leq e_1+3\leq 16+3=19 $, and so $ V_2 $ generates a valid replacement.


\begin{figure}[h]
\centering
\includegraphics[width=1\linewidth]{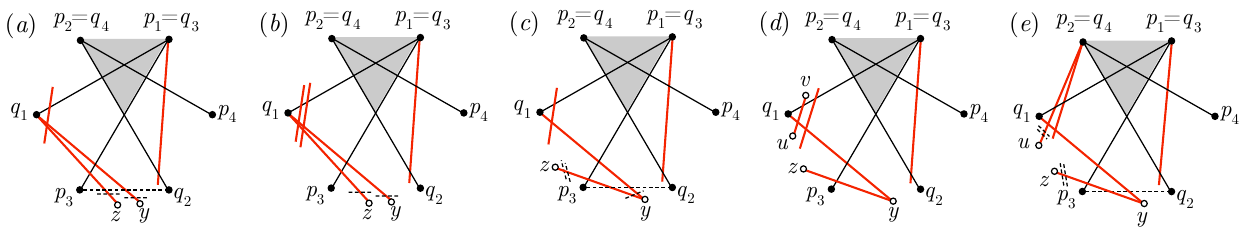}
\caption{Possible cases for Figure \ref{fig:Casej4}b with one edge in $ D $ crossing more than one edge in $ E $ and $ n_1=8 $.}
\label{fig:Casej4_1_8}
\end{figure}

If $ n_1=8 $ and $ e_1\geq 14 $, then there are two vertices $ y $ and $z $ not in $ V $ such that either $ b_3=q_1y $ and $ c_3=q_1z $ (Figure \ref{fig:Casej4_1_8}a-b), or $ b_3=q_1y $ and $ c_3=zy $ (Figures \ref{fig:Casej4_1_8}c-e). If only one edge $ b_1 \in D $ crosses $ a_1 $ (Figures \ref{fig:Casej4_1_8}a and \ref{fig:Casej4_1_8}c), then $ e'=10 $ and so $ p_3q_2 $ is an edge. But then $ e_1\leq 10+3=13 $. If two edges $ b_1 $ and $ c_1 $ cross $ a_1 $, then at most one edge not in $ E' $ crosses $ b_3 $. In Figure \ref{fig:Casej4_1_8}b, the same holds for $ c_3 $ but then $e_1\leq 11+2=13 $. In Figure \ref{fig:Casej4_1_8}d, the edge $ c_3 $ could be crossed by as many as 3 edges not in $ E' $ and so $ e_1\leq 11+3=14 $. Since $e_1\geq 14$, $ c_3 $ is crossed by 3 edges not in $ E' $, $ e'=11$, and $e_1=14$. This means that $ b_1 $ and $ c_1 $ do not cross $ c_3 $. If $ b_1=uv $ and both $ u $ and $ v $ are not in $ V $ (Figure \ref{fig:Casej4_1_8}d), then let $ V_2=V\cup\{u,v\} $. Thus $ |V_2|=8 $, $ |\gen(V_2)|\leq 11+2=13 $, and so $ V_2 $ generates a valid replacement. Since $ b_1 $ and $c_1$ cannot be incident to $ p_3 $ (as they would cross $ c_3 $), then assume that $ b_1 $ and $c_1$ are incident to $ p_2 $. Let $ b_1=p_2u $ with $ u \notin V_1 $ as in Figure \ref{fig:Casej4_1_8}e ($ c_1=p_2z$ is possible).

If $ V_3=V_1\cup \{u\} $, then $ E_3=\gen(V_3) $ is the union of $ E_1 $ and any other edges crossing $ b_1 $, $ |V_3|=9 $, and $ |E_3|\leq 14+2=16 $. If $ |E_3| \leq 15$, then $ V_3 $ generates a valid replacement. If $ |E_3|=16 $, then no edge simultaneously crosses $b_1$ and $c_3$. In particular $ q_1p_3 $ is not an edge and so $ p_3q_2 $ is an edge. Moreover, one of the other two edges not in $E'$ crossing $c_3$ is incident to a vertex not in $V_3$, call this edge $b$. Let $ V_4 $ be the union of $ V_3 $ and any vertices spanned by $ b $. Then $ |V_4|\geq 10 $, $ |\gen(V_4)|\leq 16+3=19 $, and so $ V_4 $ generates a valid replacement.




\paragraph{Case 3.} Each edge in $ D $ crosses exactly one edge in $ E $.

We  consider 9 subcases, listed on Table \ref{table:subcases}, according to the numbers $ (r_1,r_3,r_2,r_4) $ of edges in $ D $ crossing $ (a_1,a_3, a_2,a_4) $, respectively.

\begin{table}[h]
\begin{center}
\begin{tabular}{|c|c||c|c||c|c|}
\hline
Subcase $ (r_1,r_3,r_2,r_4)  $ & $ e' $&  Subcase $ (r_1,r_3,r_2,r_4)  $ & $ e' $ & Subcase $ (r_1,r_3,r_2,r_4)  $ & $ e' $\\
\hline
$ (2,2,2,2) $ & 12 & $ (2,2,2,0) $  & 10 & $ (2,1,2,1) $ & 10\\
\hline
$ (2,2,2,1) $ & 11 & $ (2,2,1,1) $  & 10 & $ (2,2,0,2) $ & 10\\
\hline
$ (2,2,1,2) $ & 11 & $ (1,2,2,1) $  & 10 & $ (2,1,1,2) $ & 10\\
\hline
\end{tabular}
\end{center}
\caption{ Division into subcases according to how the edges in $ D $ cross the edges in $ E $.}
\label{table:subcases}
\end{table}

\begin{figure}[h]
\centering
\includegraphics[width=.8\linewidth]{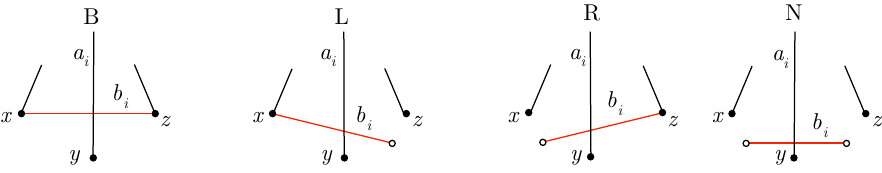}
\caption{Classification of edges in $ E $ crossed by one edge in $ D $.}
\label{fig:Casej4_types1}
\end{figure}

\paragraph{Edge-classification} We classify the edges of $ E $ according to how they are crossed by the edges of $ D $. Let $ a_i\in E $. 

\begin{itemize}
\item We say that $ a_i $ is of Type O if it does not cross edges in $ D $.

\item Figure \ref{fig:Casej4_types1} shows the possible types of $ a_i $ if it is crossed by exactly one edge $ b_i\in D $. The edge $ a_i $ is of type B, L, R, or N according to whether $ b_i $ is incident to \emph{both} $ x $ and $ z $, only $ x $ (the vertex directly to the \emph{left} of $y$ in $p_2,q_1,p_3,q_2,p_4,q_3$), only $ z $ (the vertex to the \emph{right}), or \emph{neither} $ x $ nor $ z $. 

\item Figure \ref{fig:Casej4_types2} shows the possible types of $ a_i $ if it is crossed by two edges $ b_i $ and $ c_i $ in $ D $: Types $LL$ (left-left), $LB$ (left-both), $LN$ (left-none), $RR$ (right-right), $RB$ (right-both), $RN$ (right-none), $BN$ (both-none), $NN$ (none-none). Note that type $ BB $ is not possible because we would have a double edge, and type $ LR $ is not possible because $ b_i,c_i, $ and $ a_i $ would form a triangle in $ G^{\otimes} $. The absence of triangles is $ G^{\otimes} $ also eliminates additional drawings of types LN, RN, BN, and NN.

\item We group the types of edges into four (non-disjoint) big groups: B-, which includes any of the  types B, LB, RB, and BN; R- including the types R, RB, RR, RN;  L- including the types L, LB, LL, and LN; and N- including the types N, LN, RN, BN and NN.
\end{itemize}

\begin{figure}[h]
\centering
\includegraphics[width=.8\linewidth]{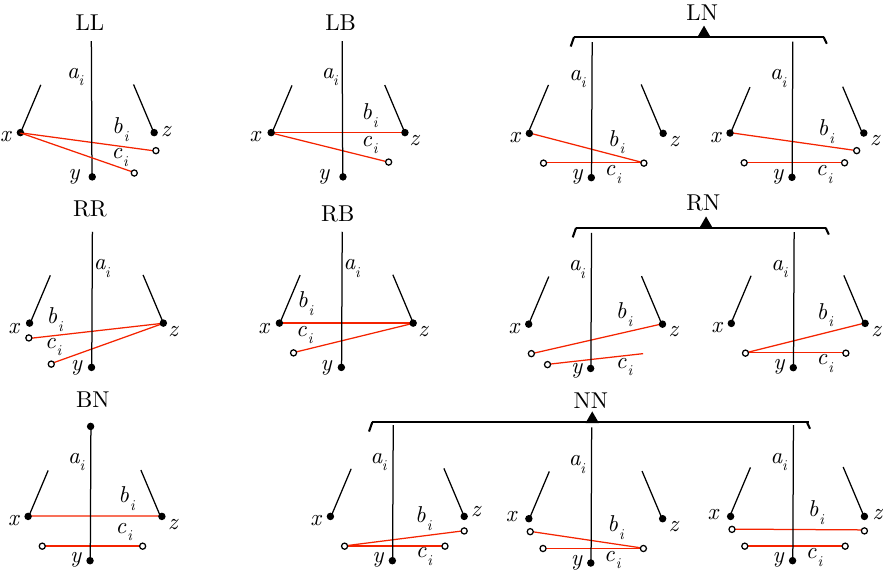}
\caption{Classification of edges in $ E $ crossed by two edges in $ D $.}
\label{fig:Casej4_types2}
\end{figure}

The following lemma identifies the incompatibility of certain types of edges.

\begin{lem}\label{obs:gen1} The graph $ G $ satisfies the following two conditions.
\begin{enumerate}
\item \label{obs:2middle_edges} The edges $ (a_3,a_2) $ cannot be of types (R- or B-,N- or R-) or (N- or L-,L- or B-).
\item \label{PNaa_PLaa} Any two consecutive edges in the list $ a_1,a_3,a_2,a_4 $ cannot be of types (RR or RB,N- or R-) or (N- or L-,LL or LB).
\item \label{obs:2B2} Within the list $ a_1,a_3,a_2,a_4 $, an edge of type B- cannot be between two edges crossed by two edges in $ D $.
\end{enumerate}
\end{lem}

\begin{proof}
\begin{enumerate}
\item Whenever the edges $ (a_3,a_2) $ are of types (R-,N-), (R-,R-), (B-,N-), (B-,R-), (N-,L-), (N-,B-), (L-,L-), or (L-,B-) a quadrilateral like the one in Figure \ref{fig:Casej4N}b is formed (Figure \ref{fig:Casej4_gen2}a).
\item Whenever two consecutive edges $a_i$ and $a_j$ in the list $ a_1,a_3,a_2,a_4 $ are of types (RR or RB,N- or R-) or (N- or L-,LL or LB), a quadrilateral like that in Figure \ref{fig:Casej4N}b is formed (Figure \ref{fig:Casej4_gen2}b).
\item Suppose $ a_i, a_j, $ and $ a_k $ are consecutive in the list, $ a_j $ is of type B-, and $ a_i $ and $ a_k $ are each crossed by two edges in $ D $.  Then $ b_j $ is crossed by $ b_i,c_i,a_j,b_k, $ and $ c_k $ (Figure \ref{fig:Casej4_gen2}c), which is impossible since any edge in $ G $ is crossed at most 4 times.
\end{enumerate}
\end{proof}

\begin{figure}[h]
\centering
\includegraphics[width=.9\linewidth]{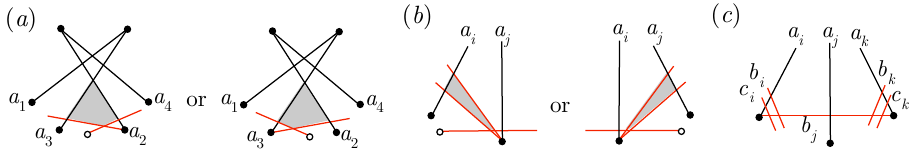}
\caption{Illustration for Lemma \ref{obs:gen1}. (a) Types (R- or B-,N- or R-) or (N- or L-,L- or B-) of ($a_3$, $a_2$) are not possible. (b) Types of any two consecutive edges (RR or RB,N- or R-) or (N- or L-,LL or LB) are not possible. (c) Types of any three consecutive edges (any 2, B-, any 2) are not possible.}
\label{fig:Casej4_gen2}
\end{figure}

The following lemma guarantees the existence of a valid replacement under certain conditions.

\begin{lem}\label{lem:valid_replacement}
	Suppose that there is a set of edges $ D'' $ in $ D $ incident to at least $ n'' $ vertices not in $ V $ and crossed by at most $ e'' $ edges not in $ E' $. If $ (5/2)(n''+3)+\delta_4(n''+6)\geq e'+e'' $, then there is a valid replacement.
\end{lem}

\begin{proof}
	Let $ D''\subseteq D $, $ V'' $ be the set of vertices not in $ V $ incident to an edge in $ D'' $, $ E'' $ be the set of edges not in $ E' $ crossing at least one edge in $ D'' $, $ |V''|=n'' $, and $ |E''|=e'' $. Then $ \gen(V\cup V'')=E'\cup E'' $. Indeed, any edge $ b $ in $ \gen(V\cup V'')-E' $ does not cross edges in $ E $ and so it does not cross sides of $\bd (V) $. So (in $ \bd (V\cup V'') $) $ b $ separates at least one vertex $ w $ in $ V'' $ from $ \bd (V) $, except perhaps by one of its endpoints. Because $ w\in V'' $, there is an edge $ c\in D'' $ incident to $ w $. Because $ c\in D'' \subseteq D $, then $ c $ crosses an edge in $ E $ and so $c$ also crosses $ \bd (V) $. Therefore, $ c $ must cross $ b $ and so $ b\in E'' $. Thus $ |V \cup V''|\geq n''+6 $ and	
	\begin{eqnarray}
		|\gen (V\cup V'')|=|E\cup E''|=e'+e''\leq \tfrac{5}{2}(n''+3)+\delta_4(n''+6)\nonumber\\
		\leq \tfrac{5}{2}(|V\cup V''|-3)+\delta_4(|V\cup V''|)+\nonedge(V\cup V'').\nonumber 
	\end{eqnarray}
	That is, $ V\cup V'' $ generates a valid replacement.
\end{proof}

\begin{cor}\label{cor:valid_replacement}
There is a valid replacement whenever there is a set $ D''\subseteq D $ incident to $ n'' $ vertices not in $V$ and crossed by $ e'' $ edges not in $E'$ satisfying one of the following conditions.
	\begin{table}[h]
		\begin{center}
		\begin{tabular}{|c|c|c|c|c|c|c|c|c|c|c|c|}
			\hline	$e'\leq$ & 12 & 12 & 12 & 11 & 11 & 11 & 11 & 10 & 10 & 10 & 10 \\
			\hline
			$n''\geq$ & 2 & 3 & 4 & 1 & 2 & 3 & 4 & 1 & 2 & 3 & 4 \\
			\hline				
			$e''\leq$ & 1 & 3 & 7 & 0 & 2 & 4 & 8 & 1 & 3 & 5 & 9\\
		    \hline	
		\end{tabular}
		\end{center}
	\end{table}
\end{cor}
Table \ref{table:compatibility11} shows Properties $ \PRRany $-$ \PV $ on $ G $ under which a valid replacement is guaranteed. The existence of these valid replacements are guaranteed by Corollary \ref{cor:valid_replacement}. Vertices not in $ V $ are indicated by empty vertices, the edges in $ D'' $ are those shown in the corresponding picture that are incident to an empty vertex. The edges in $ E'' $ are indicated by dashed lines. Bounds for the values of $ e',n'', $ and $ e'' $ are shown in the last three columns.

\begin{longtable}[t]
	   {|>{\centering\arraybackslash}m{.15in}
		|>{\centering\arraybackslash}m{.64in}
		|>{\centering\arraybackslash}m{.64in}
		|>{\centering\arraybackslash}m{.55in}
		|>{\centering\arraybackslash}m{2.2in}
		|>{\centering\arraybackslash}m{.2in}
		|>{\centering\arraybackslash}m{.23in}
		|>{\centering\arraybackslash}m{.23in}|}
\hline
P & 
Type of $ a_j $ & 
Type of $ a_i $ & 
Type of $ a_k $ & 
Figure & 
{\small $ e'\leq $} &
{\small $ n''\geq $} &
{\small $ e''\leq $} \\
\hline
\PRRany $ ^* $
&\begin{tabular}{@{}c@{}c@{}}  \\ RR \vspace{.05 in}\\  {\small any in}\vspace{-.05 in} \\ {\small Fig. \ref{fig:Casej4_types2}}\end{tabular} 
&\begin{tabular}{@{}c@{}c@{}}{\small any in}\vspace{-.05 in}\\ {\small Fig. \ref{fig:Casej4_types2}} \vspace{.05 in}\\  LL\\ \end{tabular} 
&
&\includegraphics[width=.8\linewidth]{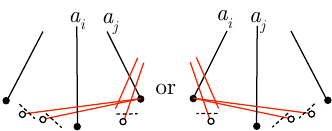} 
& 12 & 3 & 3\\
\hline
\begin{tabular}{@{}c@{}c@{}}  \PNNd \vspace{-.07 in}\\  \vspace{-.07 in} \\ \vspace{.08 in} \\\PNdN $ ^* $  \\\end{tabular}
&\begin{tabular}{@{}c@{}c@{}}NN\vspace{-.07 in}\\{\small vertex}\vspace{-.07 in}\\ {\small disjoint}\vspace{.08 in} \\ N-\end{tabular}
&&\begin{tabular}{@{}c@{}c@{}} \vspace{-.07 in}\\ \vspace{-.07 in}\\  \vspace{.08 in} \\N-\end{tabular}& \includegraphics[width=.8\linewidth]{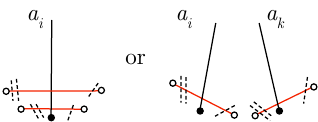} & 12 & 4 & 6 \\
\hline
\PNa
&\begin{tabular}{@{}c@{}c@{}}R-  or B-\\  \\ N-\end{tabular}&\begin{tabular}{@{}c@{}c@{}}N-\\   \\ L- or B-\end{tabular}& &\includegraphics[width=.8\linewidth]{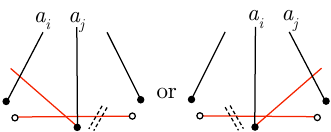} & 11 & 2 & 2 \\
\hline
\PanyLa
&\begin{tabular}{@{}c@{}c@{}}{\small any in}\vspace{-.05 in}\\ {\small Fig. \ref{fig:Casej4_types2}} \vspace{.05 in}\\  R- or B-\\ \end{tabular} 
&\begin{tabular}{@{}c@{}c@{}} \\L-  \vspace{.05 in}\\  R- \\ \end{tabular}
&\begin{tabular}{@{}c@{}c@{}} \\ L- or B- \vspace{.05 in}\\ {\small any in}\vspace{-.05 in} \\ {\small Fig. \ref{fig:Casej4_types2}}\end{tabular} 
&\includegraphics[width=.8\linewidth]{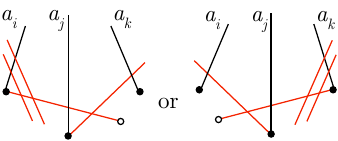} 
& 11 & 1 & 0\\
\hline
\PaRRa
& R- or B-&LL or RR& L- or B-
& \includegraphics[width=.8\linewidth]{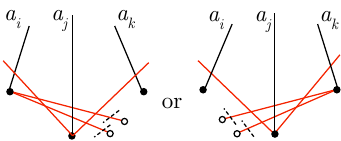} 
& 11 & 2 & 2\\
\hline
\PN 
& N-&   & 
& \includegraphics[width=.29\linewidth]{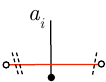} & 10 & 2 & 3\\
\hline
\PaLL
& B- or R- & LL or RR & 
 & \includegraphics[width=1\linewidth]{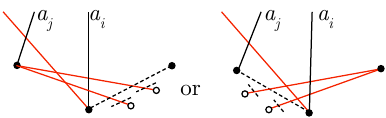} & 10 & 2 & 3\\
\hline
\PLLa 
& LL or RR & B- or L-& 
 & \includegraphics[width=1\linewidth]{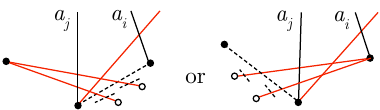} & 10 & 2 & 3\\
\hline
\PanyL
 &\begin{tabular}{@{}c@{}c@{}}{\small any in}\vspace{-.05 in}\\ {\small Fig. \ref{fig:Casej4_types2}}\vspace{.05 in} \\ R-\\ \end{tabular}
&\begin{tabular}{@{}c@{}c@{}}  \\ L- \vspace{.05 in} \\  {\small any in}\vspace{-.05 in} \\ {\small Fig. \ref{fig:Casej4_types2}}\end{tabular} 
&
&\includegraphics[width=.75\linewidth]{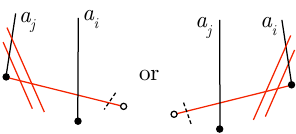} & 10 & 1 & 1\\
\hline
\PV
& B- or  R- & L- or R-  & B- or L- & \includegraphics[width=.8\linewidth]{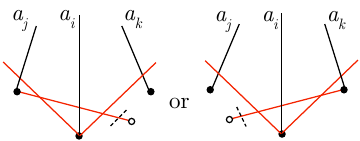} & 10 & 1 & 1\\
\hline
\caption{Properties $ \PRRany $-$ \PV $ guarantee the existence of a valid replacement. The edges $ a_i $, $ a_j $, and $ a_k $ are consecutive in $ \{a_1,a_3,a_2,a_4\} $, except for Property $ \PNdN $ where $ a_i $ and $ a_k $ could be any edges in $ \{a_1,a_2,a_3,a_4\} $ as long as the edges in $ D'' $ do not share vertices. In Property $ \PRRany $, note that at least one of the two edges crossing $ a_j $ is incident to a vertex not in $ V $, and this vertex could be above $a_j$.}
\label{table:compatibility11}
\end{longtable}
 



We now analyze each of the subcases in Table \ref{table:subcases}.

\paragraph{Subcase (2,2,2,2).} In this subcase $ e'=12 $ and by Lemma \ref{obs:boundary} we assume that there are at most 2 nonedges on $ \partial\conv(V) $.

By Lemma \ref{obs:gen1}.\ref{obs:2B2}, $ a_3 $ and $ a_2 $ cannot be of type B-. Property $\PRRany$ guarantees a valid replacement when at least one of the edges $ a_3 $ or $ a_2 $ is of type RR or LL. Then assume that $ a_3 $ and $ a_2 $ are of types LN, RN, or NN. By Lemma \ref{obs:gen1}.\ref{obs:2middle_edges}, only the cases when $ a_3 $ is of type LN or NN and $ a_2 $ is of type RN or NN are left. By Properties $ \PNdN $ and Lemma \ref{obs:gen1}.\ref{PNaa_PLaa}, the only cases left are when $ a_1 $ is of type LL or LB and $ a_4 $ is of type RR or RB. Since $ a_3 $ and $ a_2 $ are of type N-, Property $ \PNNd $ guarantees a valid replacement unless $ c_3 =xy $ and $ c_2=yz $ for some vertices $ x,y, $ and $ z $ not in $ V $. Assuming this situation and in order to avoid triangles or quadrilaterals like the one in Figure \ref{fig:Casej4N}a-b in $ G^{\otimes} $, any edges crossing $ a_3 $ or $ a_2 $ must be incident to $ y $, Figure \ref{fig:Case2222}a. 

\begin{figure}[h]
\centering
\includegraphics[width=.8\linewidth]{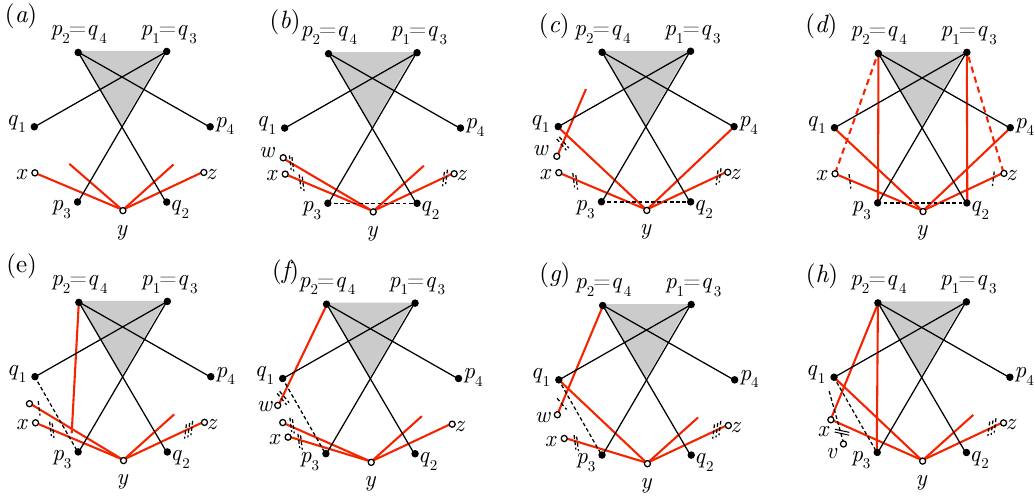}
\caption{Possible subgraphs for Subcase (2,2,2,2).}
\label{fig:Case2222}
\end{figure}

Because $ q_1p_3,p_3q_2, $ and $ q_2p_4 $ are not sides (they are crossed by $ c_3 $ or $ c_2 $), then at least one of them is an edge by Lemma \ref{obs:boundary}. Suppose first that $ p_3q_2 $ is an edge. If $ a_3 $ (or $a_2$) is of type NN, Figure \ref{fig:Case2222}b, then $ b_3=wy$ for some  vertex $ w\notin V\cup \{x,y,z\}$. Let $ D''=\{b_3,c_3,c_2\} $. These three edges are incident to $ w,x,y, $ and $ z $ and (because they are simultaneously crossed by $ p_3q_2 $) they are crossed by at most $ 1+2+2+2=7 $ edges not in $ E' $. So Corollary \ref{cor:valid_replacement} guarantees a valid replacement with $n"=4$ and $e''\leq 7$. So assume that $ a_3 $ and $ a_2 $ are of types LN and RN. If $ b_1 $ (or $b_4$) is incident to a vertex $ w \notin V\cup \{x,y,z\} $, Figure \ref{fig:Case2222}c ($w$ could be above $q_1$), then use Corollary \ref{cor:valid_replacement} with $ D''=\{b_1,c_3,c_2\} $, $n"=4$, and $e''\leq 7$. So assume that $ a_1 $ and $ a_4 $ are of types B-, Figure \ref{fig:Case2222}d. Corollary \ref{cor:valid_replacement} guarantees a valid replacement using $ D''=\{c_3,c_2\}, $ $n"=3$, and $e''\leq 3$. 

Now assume that $ q_1p_3 $ is an edge and $ p_3q_2 $ is a nonedge. Suppose first that $ a_3 $ is of type NN. If $ b_1 $ or $c_1$ crosses $ b_3 $ or $ c_3 $, let $ D''=\{b_3,c_3,c_2\} $, Figure \ref{fig:Case2222}e. Otherwise, $ b_1 $ (or $c_1$) does not share vertices with $b_3$ and $c_3$; let $ D''=\{b_1,b_3,c_3\} $ as in Figure \ref{fig:Case2222}f. In both cases $ D'' $ is incident to 4 vertices not in $ V $ and is crossed by at most 7 edges not in $ E' $. So Corollary \ref{cor:valid_replacement} guarantees a valid replacement. Suppose now that $ a_3 $ is of type LN. If $ b_1 $ or $c_1$ is incident to a vertex $ w\notin V \cup\{x\}$, Figure \ref{fig:Case2222}g, then use Corollary \ref{cor:valid_replacement} with $ D''=\{b_1,c_3,c_2\} $, $n"=4$, and $e''\leq 7$. Otherwise, we can assume that $ a_1 $ is of type LB and $ b_1=p_2x $. If $ q_1p_3 $ is the only edge not in $ D $ crossing $ xy $, then use Corollary \ref{cor:valid_replacement} with $ D''=\{c_3\} $, $n"=2$, and $e''\leq 1$. Otherwise, there is an edge $ uv $, with $ v\notin V $ crossing $ xy $, Figure \ref{fig:Case2222}h. Let $V''=V\cup \{v,x,y,z\}  $ and $ E''=\gen(V'') $, the union of $ E' $ and any other edges crossing $ xy,yz, $ or $ uv $. There are two edges not in $ E' $ crossing $ xy $ ($ q_1p_3 $ and $ uv $), at most three crossing $ yz $, and at most two crossing $ uv $.  Then $ |V''|=6+4=10 $, $ |E''|\leq 12+7=19 $, and thus $ V'' $ generates a valid replacement.

\paragraph{Subcases (2,2,2,1) and (2,2,1,2).} In these subcases $ e'=11 $. We can also assume that there are no edges of type NN. Indeed, suppose there is an edge $ a_i\in E $ of type NN. The edges $ b_i $ and $ c_i $ do not cross because they would form a triangle in $ G^{\otimes} $. Property $ \PNNd $ guarantees the existence of a valid replacement if $ b_i $ and $ c_i $ are disjoint.  Assume then that $ b_i=xy $ and $ c_i=xz $ for some vertices $ x,y, $ and $ z $ not in $ V $, Figure \ref{fig:Case2221}a. By Lemma \ref{obs:boundary}, we can assume that one of the diagonals in $ \bd(V) $ crossed by $ b_i $ is an edge. This edge simultaneously crosses $ b_i $ and $ c_i $ and thus $ b_i $ and $ c_i $ can be crossed by at most 5 edges not in $ E' $. If there is an edge $b_j$ in $D'$ incident to a vertex $w\notin V\cup\{x,y,z\}$, then use Corollary \ref{cor:valid_replacement} with $ D''=\{b_i,c_i,b_j\} $, $ n''\geq 4$,  and $e''\leq 5+3=8 $. Such an edge exists when $i=1$ (and by symmetry, when $i=4$) because one of the edges $a_2$ or $a_4$ is crossed by 2 edges in $D$. So assume that 
$i=3$ (symmetric argument if $i=2$), then $a_4$ must be of type B as in Figure \ref{fig:Case2221}b. Thus $a_2$ is crossed by two edges in $D$, both of them incident to $p_4$.  One of these edges must cross $b_3$ or $c_3$. In this case, use Corollary \ref{cor:valid_replacement} for $ D''=\{b_3,c_3\} $, $ n''\geq 3 $, and $ e''\leq 5-1=4 $.


\begin{figure}[h]
\centering
\includegraphics[width=.8\linewidth]{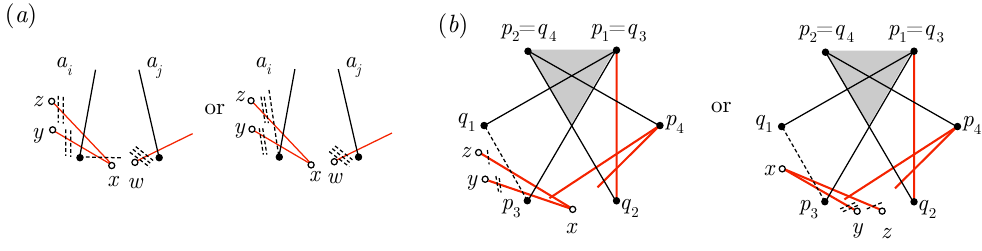}
\caption{(a) There are no edges of type NN in subcases (2,2,2,1) and (2,2,1,2), the vertices $x,y,z$ could be on the other side of $a_i$ and $w$ could be on the other side of $a_j$. (b) A possible subgraph for Subcase (2,2,2,1), the edge $q_1p_3$ could be replaced by the edge $p_3q_2$.}
\label{fig:Case2221}
\end{figure}

In \textbf{Subcase (2,2,2,1)}, $ a_3 $ cannot be of type B- by Lemma \ref{obs:gen1}.\ref{obs:2B2}, or RN by Lemma \ref{obs:gen1}.\ref{obs:2middle_edges}; and there is a valid replacement when $ a_3 $ is of type RR or LL by Property $ \PRRany $. So we assume that $ a_3 $ is of type LN and $b_3=q_1x$.  Then $ a_2 $ is not of type L-, or B- by Lemma \ref{obs:gen1}.\ref{obs:2middle_edges}, and there is a valid replacement when $ a_1 $ is of type N-, R-, or B- by Property $\PNa$. So we assume that $ a_1 $ is of type LL and $ a_2 $ is of type RR or RN; with $b_1=p_2y$ and $c_1=p_2z$. 
Because one of the edges crossing $ a_2 $, say $ b_2 $, is incident to a vertex $ w\notin V\cup\{x,y,z\}$, then Corollary \ref{cor:valid_replacement} using $ D''=\{b_1,c_1,b_3,b_2\} $, $ n''\geq 4 $, and $ e''\leq 8 $ guarantees a valid replacement. 

\begin{figure}[h]
\centering
\includegraphics[width=.8\linewidth]{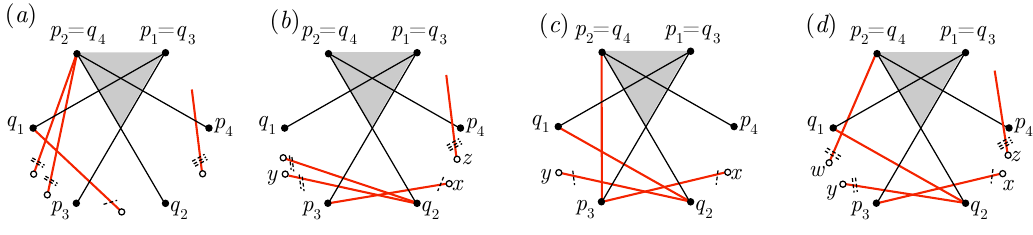}
\caption{Possible subgraphs for Subcase (2,2,1,2).}
\label{fig:Case2212}
\end{figure}

In \textbf{Subcase (2,2,1,2)}, $ a_3 $ cannot be of type RN, BN, or LB by Lemma \ref{obs:gen1}.\ref{obs:2middle_edges} and there is a valid replacement if $ a_3 $ is of type LL by Property $ \PRRany $. So we assume that $ a_3 $ is of type LN, RR, or RB.  

If $ a_3 $ is of type LN, then there is a valid replacement unless $ a_1 $ is of type LL by Property $ \PNa $. Since one of the edges crossing $ a_4 $, say $ b_4 $, is incident to a vertex not in $ V $, then there is a valid replacement by Corollary \ref{cor:valid_replacement} using $ D''=\{b_1,c_1,b_3,b_4\} $ ($ b_3 $ is incident to $ q_1 $), $ n''\geq 4 $, and $ e''\leq 8 $.

If $ a_3 $ is of type RR or RB, then $ a_2 $ must be of type L by Lemma  \ref{obs:gen1}.\ref{obs:2middle_edges}-\ref{obs:2B2}. Then $b_2=p_3x$ and $b_3=q_2y$ for some vertices $x$ and $y$ not in $V$. By Property $ \PanyLa $, a valid replacement exists if $ a_4 $ is of type L- or B-. So assume that $ a_4 $ is of type RR or RN, and thus one of the edges crossing $ a_4 $, say $ b_4 $, is incident to a vertex $z\notin V\cup\{x,y\} $. If $ a_3 $ is of type RR, then there is a valid replacement by Corollary \ref{cor:valid_replacement} using $ D''=\{b_3,c_3,b_2,b_4\} $, $ n''\geq 4 $, and $ e''\leq 8 $. If $ a_3 $ is of type RB, then there is a valid replacement when $ a_1 $ is of type RR or N- by Properties $ \PRRany $ and $ \PNa $, respectively. If $ a_1 $ is of type B-, then $ b_3 $ and $ b_2 $ are crossed by two edges in $ D' $ and so Corollary \ref{cor:valid_replacement} can be used with $ D''=\{b_3,b_2\} $, $ n''\geq 2 $, and $ e''\leq 2 $. So assume that $ a_1 $ is of type LL. Since one of the edges crossing $ a_1 $, say $ b_1 $, is incident to a vertex $w\notin V\cup\{x,y,z\} $, then there is a valid replacement by Corollary \ref{cor:valid_replacement} using $ D''=\{b_1,b_3,b_2,b_4\} $, $ n''\geq 4 $, and $ e''\leq 8 $.


\paragraph{All other Subcases.} In all the remaining subcases, $ e'=10 $, by Lemma \ref{obs:boundary} we assume that $ \partial \conv(V) $ is formed by sides and edges only, and by Property $ \PN $ we assume that there are no edges of type N-. The following corollary follows from  Properties $ \PN$-$ \PanyL $.
\begin{cor}\label{cor:22}
Let $ a_i $ and $ a_j $ be two consecutive edges in the list $ a_1,a_3,a_2,a_4 $, each crossed by two edges in $ D $. Then there is a valid replacement unless $ a_i $ and $ a_j $ are of types (in this order) LB-RB or LL-RR.
\end{cor}

In \textbf{Subcase (2,2,2,0)}, there are three consecutive 2s. If $ a_1 $-$ a_3 $ are of types LB-RB or LL-RR, then $ a_3 $-$ a_2 $ are not of those types and thus a valid replacement is guaranteed by Corollary \ref{cor:22}. 

In \textbf{Subcases (2,2,1,1) and  (1,2,2,1)}, there is a 2,2,1 substring. By Corollary \ref{cor:22}, there is a valid replacement unless the first two of these three edges are of types LB-RB or LL-RR.  If the third edge is of type R- or L-, then Lemma \ref{obs:gen1}.\ref{PNaa_PLaa} and Property $ \PanyL $, respectively, guarantee a valid replacement. If the third edge is of type B, then LL-RR-B is covered by Property $ \PLLa $ and LB-RB-B is covered by Property $ \PV $. 
\begin{figure}[h]
\centering
\includegraphics[width=.5\linewidth]{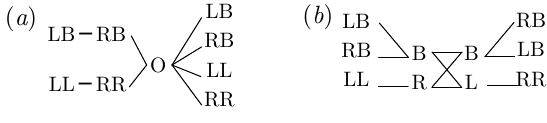}
\caption{Remaining possibilities for Subcases (2,2,0,2) and (2,1,1,2).}
\label{fig:Casej4_bd}
\end{figure}

In \textbf{Subcase (2,1,2,1)}, $ a_3 $ is not of type B by Lemma \ref{obs:gen1}.\ref{obs:2B2}. The remaining possibilities for $ a_3 $ are covered by Property $ \PanyL $.

\begin{figure}[h]
\centering
\includegraphics[width=.8\linewidth]{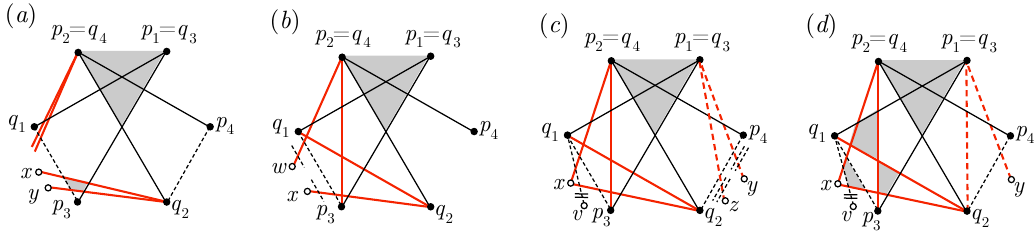}
\caption{Possible subgraphs for Subcase (2,2,0,2).}
\label{fig:Casej4_b}
\end{figure}

In \textbf{Subcase (2,2,0,2)}, by Property $\PN$ and Corollary \ref{cor:22}, the only remaining possibilities for $(a_1,a_3,a_2,a_4)$ are the left-to-right paths in Figure \ref{fig:Casej4_bd}a. Figure \ref{fig:Casej4_b}a corresponds to $ a_1 $-$ a_3 $ being of types LL-RR and so we assume that $ q_1p_3 $ is an edge. If $ q_1p_3 $ is the only edge not in $ E' $ crossing $ b_3 $ (or $ c_3 $), then use Corollary \ref{cor:valid_replacement} for $ D''=\{b_3\} $, $ n''\geq 1 $, and $ e''= 1 $. Otherwise, the quadrilateral in $ G^\otimes $ formed by the set of edges $E''= \{a_3,q_1p_3,b_3,c_3\} $ is like that in Figure \ref{fig:Casej4}b and is crossed by at least 6 other edges. So it must fall into one of the subcases in Table \ref{table:subcases}. We have shown the existence of a valid replacement for all subcases other than (2,2,0,2), (2,0,2,2), and (2,1,1,2). Since $ a_3 $ is crossed by $ a_2 $ and $ a_4 $, $ q_1p_3 $ is crossed by $ b_1 $ and $ c_1 $, and $ b_3 $ and $ c_3 $ are crossed by an edge other than $ a_3 $ and $ q_1p_3 $, then the quadrilateral cannot fall into any of these three subcases and so it has a valid replacement. 


Figures \ref{fig:Casej4_b}b-d correspond to $ a_1 $ and $ a_3 $ being of types LB-RB and so we assume that $ q_1p_3 $. Since $a_4$ is of type R- or L-, then $ q_2p_4 $ or $ p_4q_3 $, respectively, is also an edge. Let $ b_1=p_2w $ and $ b_3=q_2x $ with $ w $ and $ x $ not in $ V $. If $ w\neq x $  as in Figure \ref{fig:Casej4_b}b, use Corollary \ref{cor:valid_replacement} for $ D''=\{b_1,b_3\} $, $ n''\geq 2 $, and $ e''\leq 3 $. So assume that $ w=x $. If $ q_1p_3 $ is the only edge not in $ E' $ crossing $ b_1 $, then use Corollary \ref{cor:valid_replacement} for $ D''=\{b_1\} $, $ n''\geq 1 $, and $ e''\leq 1 $. So assume that there is another edge $ uv $ crossing $ b_1 $. Note that this edge must be incident to $q_1$ or $p_3$ in order to avoid a quadrilateral like that in Figure \ref{fig:Casej4N}b. So we assume that $v\notin V$ and $u=p_1$ or $u=q_3$. If $a_4$ is of type LL (or similarly RR) with $b_4=p_1y$ and $c_4=p_1z$, see Figure \ref{fig:Casej4_b}c, let $V'=V\cup \{v,x,y,z\}$. Then $|V'|=10$, $|\gen(V')|\leq 10+9=19$, and thus $V'$ generates a valid replacement. If $a_4$ is of type LB or RB, Figure \ref{fig:Casej4_b}d, then we can assume that the edge $uv$ is crossed by two edges not in $E'$ as otherwise the set $V'=V\cup \{v,x\}$ with $|V'|=8$, $|\gen(V')|\leq 10+3=13$ would generate a valid replacement. This is the only remaining possibility for which we have not yet found a valid replacement in this subcase and it will be analyzed together with the subcase (2,1,1,2), see Figure \ref{fig:Casej4_left}a.

\begin{figure}[h]
	\centering
	\includegraphics[width=1\linewidth]{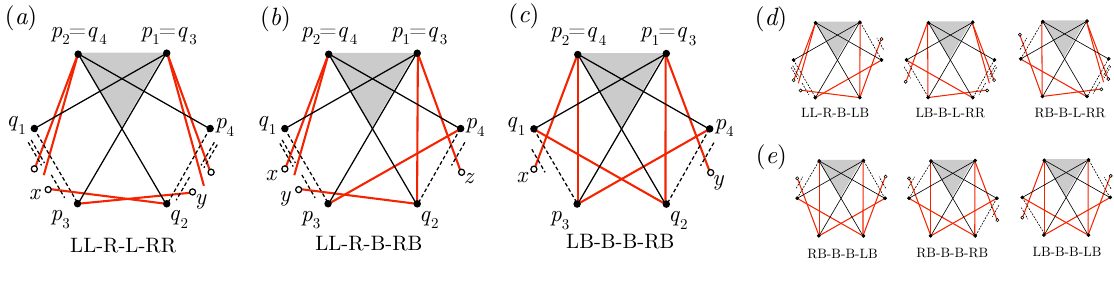}
	\caption{Possible subgraphs for Subcase (2,1,1,2).}
	\label{fig:Casej4_e}
\end{figure}

In \textbf{Subcase (2,1,1,2)}, by Properties $ \PN $-$ \PV $ and Lemma \ref{obs:gen1}.\ref{PNaa_PLaa}, the only remaining possibilities are the left-to-right paths in Figure \ref{fig:Casej4_bd}b. All these cases are illustrated in Figure \ref{fig:Casej4_e}. The arguments for Figures \ref{fig:Casej4_e}d and \ref{fig:Casej4_e}e are analogous to those of Figures \ref{fig:Casej4_e}b and \ref{fig:Casej4_e}c, respectively. 

In Figure \ref{fig:Casej4_e}a, let $ x $ and $ y $ be the vertices not in $ V $ incident to $ b_3 $ and $ b_2 $, respectively. Then we can assume that $ b_1 $ is incident to a vertex not in $ V\cup\{x\} $; and $ b_4 $ is incident to a vertex not in $ V\cup\{y\}  $. By Lemma \ref{obs:boundary}, the diagonals $ q_1p_3 $ and $ q_2p_4 $ are edges. Corollary \ref{cor:valid_replacement} guarantees a valid replacement using $ D''=\{b_1,b_3,b_2,b_4\} $, $ n''=4 $, and $ e''\leq 8 $. 

\begin{figure}[h]
\centering
\includegraphics[width=1\linewidth]{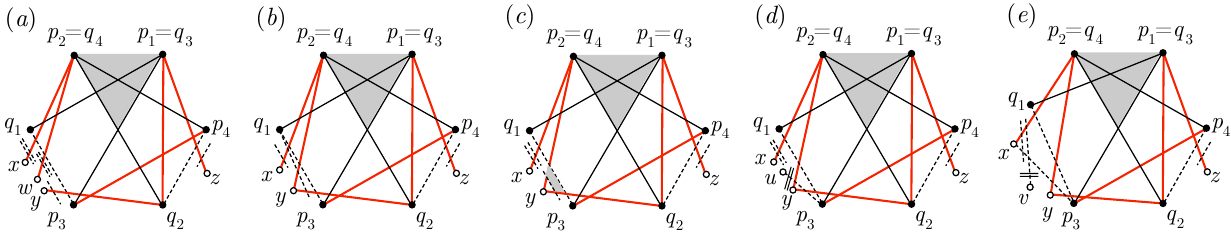}
\caption{Analysis of Figure \ref{fig:Casej4_e}b.}
\label{fig:Casej4_ext2112a}
\end{figure}

In Figure \ref{fig:Casej4_e}b, we can assume that $ q_1p_3 $ and $ q_2p_4 $ are edges by Lemma \ref{obs:boundary}. Let $ w,x,y, $ and $ z $ be the vertices not in $ V $ incident to $ b_1,c_1,b_3 $ and $ b_4 $, respectively. If $x\neq y$ and $w\neq y$, Figure \ref{fig:Casej4_ext2112a}a, then Corollary \ref{cor:valid_replacement} guarantees a valid replacement using $ D''=\{b_1,c_1,b_3,b_4\} $, $ n''=4 $, and $ e''\leq 8 $. Assume that $w=y$ and let $D''=\{b_1,c_1,b_3,b_4\}$ so that $ n''=3 $ and $ e''\leq 6 $. If $e'\leq 5$, then Corollary \ref{cor:valid_replacement} guarantees a valid replacement. So assume that $e'=6$ and so there must be an edge $uu'\neq q_1p_3$ not in $E'$ crossing $c_1$ (and $b_3$) and nonincident to $q_1$ (otherwise $e''\leq 5$, Figure \ref{fig:Casej4_ext2112a}b). If $u$ and $u'$ are not in $V$, Figure \ref{fig:Casej4_ext2112a}c, then there would be a quadrilateral like the one in Figure \ref{fig:Casej4N}b. So assume that $u'=p_3$. If $u\neq x$, Figure \ref{fig:Casej4_ext2112a}d, then $up_3$ can be crossed by at most 2 edges not in $E'$. Let $ V'=V\cup \{u,x,y,z\} $, then $|V'|=10$, $|\gen(V')|\leq 10+8=18$, and thus $V'$ generates a valid replacement. If $u=x$, Figure \ref{fig:Casej4_ext2112a}e, then one of the two edges crossing $b_1$ is incident to a vertex $v$ not in $V\cup\{x,y,z\}$ and such an edge can be crossed by at most two more edges not in $E'$ (because it is already crosses $b_1$ and $xp_3$). Let $ V'=V\cup \{v,x,y,z\} $, then $|V'|=10$, $|\gen(V')|\leq 10+8=18$, and thus $V'$ generates a valid replacement. 

In Figures \ref{fig:Casej4_e}c, we can assume that $ q_1p_3 $ and $ q_2p_4 $ are edges by Lemma \ref{obs:boundary}.  The edges $ b_1 $ and $ b_4 $ are incident to some vertices $ x\notin V $ and $ y\notin V $, respectively. If  $V'=V\cup\{x,y\}$, then $|V'|=8$ and $|\gen(V')|\leq 14$ and so $V'$ generates a valid replacement whenever $14\leq 13+|\nonedge(V')|$, that is, whenever $|\gen(V')|\leq 13$ or $|\nonedge(V')|\geq 1$. So we can assume that $|\gen(V')|= 14$ and $|\nonedge(V')|=0$. Let $ uu' $ and $ vv' $ (Figure \ref{fig:Casej4_ext2112}a) be the other edges crossing $ b_1 $ and $ b_4 $, respectively.

\begin{figure}[h]
\centering
\includegraphics[width=.65\linewidth]{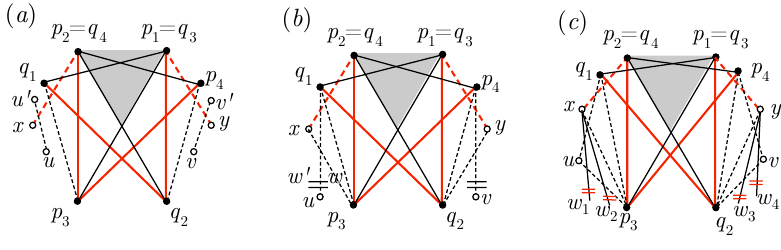}
\caption{Analysis of Figure \ref{fig:Casej4_e}c: extending a cycle 2112.}
\label{fig:Casej4_ext2112}
\end{figure}

If $ u $ and $ u' $ are not in $ V $ (similar argument for $v$ and $v'$), then let $ V_1=V\cup\{x,y,u,u'\} $. Then $ |V_1|=10 $ and $ |\gen(V_1)|\leq 10+2+2+3=17 $ and thus there is a valid replacement generated by $ V_1 $. So assume that $ u'=q_1 $ and $ v'=p_4 $ as in Figure \ref{fig:Casej4_ext2112}b (the argument is analogous if instead $ u=p_3 $ and/or $ v=q_2 $). Since $|\nonedge(V')|=0$, then $ xp_3 $ and $ yq_2 $ are edges. Let $ V_2=V\cup\{x,y,u,v\} $. Then $ |V_2|=10 $ and $ |\gen(V_2)|\leq 10+10=20 $. Hence, $V_2$ generates a valid replacement whenever $20\leq 13+|\nonedge(V_2)|$, that is, if $|\gen(V_2)|\leq 19$ or $|\nonedge(V_2)|\geq 1$. So assume that $|\gen(V_2)|=20$ and $|\nonedge(V_2)|=0$. If one of the edges $ w'w $ crossing $ uq_1 $ or $ vp_4 $ is not incident to any vertices in $ V_2 $, let $ V_3=V_2\cup\{w,w'\} $. Then $ |V_3|=12 $, $ |\gen(V_3)|\leq 20+3=23 $, and thus $ V_3 $ generates a valid replacement. Then (in order to avoid triangles in $ G^{\otimes} $) we can assume that the two extra edges crossing $ uq_1 $ are $ xw_1 $ and $ xw_2 $; and the two extra edges crossing $ vp_4 $ are $ yw_3 $ and $ yw_4 $ as in Figure \ref{fig:Casej4_ext2112}c. (The argument is analogous if these edges are incident to $p_3$ and/or $q_2$ instead of $x$ and $y$, respectively.) Since $|\nonedge(V_2)|=0$, then $ up_3 $ and $ vq_2 $ are edges. Let $ V_4=V_2\cup \{w_1,w_2,w_3,w_4\} $. Then $ |V_4|=14 $ and $ |\gen(V_4)|\leq 20+10=30 $. As long as $ |\gen(V_4)|\leq 29$, there is a valid replacement generated by $ V_4$. So assume that $ |\gen(V_4)|=30 $. This is the only remaining possibility for which we have not yet found a valid replacement in this subcase, see Figure \ref{fig:Casej4_left}b.

\begin{figure}[h]
\centering
\includegraphics[width=.75\linewidth]{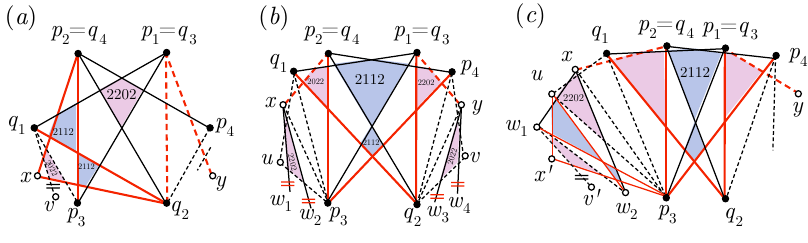}
\caption{(a-b) The only remaining subgraphs in Subcases (2,2,0,2) and (2,1,1,2), respectively. (c) Puting together Subcases (2,2,0,2) and (2,1,1,2).}
\label{fig:Casej4_left}
\end{figure}

Figures \ref{fig:Casej4_left}a-b (and equivalent cases, all of which can be obtained by reflections of one side of some of their diagonals) are all the remaining possibilities for which we have not found valid replacements. Figure \ref{fig:Casej4_left}a corresponds to Subcase (2,2,0,2) and we say that the quadrilateral formed by $ a_1,a_2,a_3, $ and $ a_4 $ is of type 2202. Figure \ref{fig:Casej4_left}b corresponds to Subcase (2,1,1,2) and we say that the quadrilateral formed by $ a_1,a_2,a_3, $ and $ a_4 $ is of type 2112. Note that the existence of a quadrilateral of type 2202 implies the existence of either a valid replacement or a quadrilateral of type 2112, and vice versa. Indeed, the quadrilaterals formed by the sets of edges $\{b_1,c_1,q_1q_2,a_1\} $, $\{a_3,p_2p_3,b_3,c_3\} $, and  $\{q_1v,q_1p_3,xq_2,xp_2\}$ in Figure \ref{fig:Casej4_left}a are of types 2112, 2112, and 2022, respectively; and the quadrilaterals formed by the sets of edges $\{p_3p_4,a_3,a_2,q_2q_1\} $,  $\{b_1,c_1,q_1q_2,a_1\} $,
$\{a_4,p_4p_3,b_4,c_4\} $, $\{xw_1,xw_2,up_3,uq_1\} $, and  $\{vp_4,vq_2,yw_3,yw_4\}$ in Figure \ref{fig:Casej4_left}b are of types 2112, 2022, 2202, 2202, and 2022, respectively.

So assume that there is a quadrilateral of type 2112 as in Figure \ref{fig:Casej4_left}b. Then either there is a valid replacement at some point or its most left quadrilateral of type 2202 can be extended as the main quadrilateral of type 2202 in Figure \ref{fig:Casej4_left}a. This is shown in Figure  \ref{fig:Casej4_left}c, where $ x,u,w_1,x',v',w_2,p_3,p_2, $ and $ q_1 $ in Figure \ref{fig:Casej4_left}c correspond to $ p_1,p_2,q_1,x,v,p_3,q_2,y,$ and $p_4 $ in Figure \ref{fig:Casej4_left}a, respectively. Now Figure \ref{fig:Casej4_left}c includes a new quadrilateral of type 2112, to which the previous argument could be applied. If we repeat the argument two more times, we can guarantee either a valid replacement at some point or the subgraph in Figure \ref{fig:Final22_48} that consists of a set $ V'' $ of 22 vertices whose convex hull interior is crossed by at most 48 edges. Then $ V'' $ generates a valid replacement.


\begin{figure}[h]
\centering
\includegraphics[width=.7\linewidth]{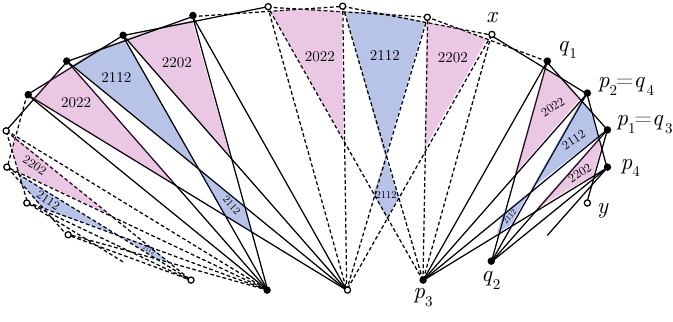}
\caption{A valid replacement with 22 vertices and 48 edges. The graph $ H_{22}=5D_6 $ has 49 edges.}
\label{fig:Final22_48}
\end{figure}
\end{proof}

\noindent {\bf Acknowledgements.} We would like to thank the referees for their valuable suggestions. The presentation was greatly enhanced because of these suggestions.

\end{document}